\documentclass[a4,12pt]{amsart}

\usepackage{amssymb}
\usepackage{amstext}
\usepackage{amsmath}
\usepackage{amscd}
\usepackage{amsthm}
\usepackage{amsfonts}
\usepackage{enumerate}
\usepackage{fancybox}
\usepackage{graphicx}
\usepackage{latexsym}
\usepackage[all,arc]{xy}
\usepackage[usenames]{color}
\usepackage{mathrsfs}
\usepackage{color}
\usepackage[T1]{fontenc}

\newenvironment{red}
{\relax\color{red}}
{\hspace*{.5ex}\relax}

\newcommand{\ber}{\begin{red}}
\newcommand{\er}{\end{red}}

\theoremstyle{definition}
\newtheorem*{maintheorem}{Theorem}

\newtheorem*{theo}{Theorem}
\newtheorem{theorem}{Theorem}[section]
\newtheorem{corollary}[theorem]{Corollary}
\newtheorem{lemma}[theorem]{Lemma}
\newtheorem{proposition}[theorem]{Proposition}

\newtheorem{definition}[theorem]{\rm Definition}
\newtheorem{remark}[theorem]{\rm Remark}

\def \para{\refstepcounter{theorem} \par\medskip\noindent
                \textbf{\thetheorem .} }

\theoremstyle{remark}

\numberwithin{equation}{section}

\renewcommand{\mod}{\operatorname{mod}}

\newcommand{\End}{\operatorname{End}}
\newcommand{\Hom}{\operatorname{Hom}}

\newcommand{\K}{\mathcal{K}}
\newcommand{\latt}{\mathsf{latt}}
\renewcommand{\O}{\mathcal{O}}
\renewcommand{\k}{\mathbf{k}}

\newcommand{\Z}{{\mathbb{Z}}}

 \def\hsymb#1{\mbox{\strut\rlap{\smash{\Huge$#1$}}\quad}}
\begin{document}

\title[components of stable AR quivers]{On the non-periodic stable Auslander--Reiten  Heller component for the Kronecker algebra over a complete discrete valuation ring}
\date{}

\author[Kengo Miyamoto]{Kengo Miyamoto}
\address{Department of Pure and Applied Mathematics, Graduate School of Information
Science and Technology, Osaka University, Suita, Osaka 565-0871, Japan}
\email{k-miyamoto@ist.osaka-u.ac.jp}


\begin{abstract}
We consider the Kronecker algebra $A=\mathcal{O}[X,Y]/(X^2,Y^2)$, where $\O$ is a complete discrete valuation ring. Since $A\otimes\kappa$  is a special biserial algebra, where $\kappa$ is the residue field of $\mathcal{O}$, one can compute a complete list of indecomposable $A\otimes \kappa$-modules. For each indecomposable $A\otimes \kappa$-module, we obtain a special kind of $A$-lattices called ``Heller lattices''.  In this paper, we determine the non-periodic component of a variant of the stable Auslander--Reiten quiver for the category of $A$-lattices that contains ``Heller lattices''.  
\end{abstract}


\maketitle
\tableofcontents

\section*{Introduction} 
Auslander--Reiten theory has become an indispensable tool since we may prove many important combinatorial and homological properties with the help of the theory, and it gives us invariants of various additive categories arising in representation theory, for example see \cite{ARS}, \cite{ASS}, \cite{H} and \cite{Y}. A combinatorial skeleton of the additive category of indecomposable objects is the Auslander--Reiten quiver, which encapsulates much information on indecomposable objects and irreducible morphisms. 
Therefore, to determine the shape of Auslander--Reiten quivers is one of classical problems in representation theory of algebras. 

There exist strong restrictions on stable Auslander--Reiten quivers for important classes of finite dimensional algebras. 
In \cite{We}, Webb studied the stable Auslander--Reiten components of group algebras. Let $G$ be a finite group and $\k$ an algebraically closed field with characteristic $p$ such that $p$ divides the order of $G$. Then, the tree class of any stable component of the group algebra $\k G$ is one of infinite Dynkin diagrams $A_{\infty}, B_{\infty}, C_{\infty}, D_{\infty}$ or $A_{\infty}^{\infty}$, or else it is $A_n$, or one of Euclidean diagrams.
Moreover, Erdmann showed that the tree class of any stable component of a wild block of $\k G$ is $A_{\infty}$ \cite{Erd}.
For another example, Riedtmann and Todorov showed that the tree class of any stable component of a finite dimensional self-injective algebra of finite representation type is one of finite Dynkin diagrams \cite{Ri2, T}. 
However, if the base ring is not a field but a regular local ring, then the shape of  (stable) Auslander--Reiten components for algebras are mostly unknown.

We use the following notation, see \cite{I} for details.
Let $\mathcal{O}$ be a complete discrete valuation ring, $\kappa$ its residue field, $\mathcal{K}$ its fraction field. An $\mathcal{O}$-algebra $A$ is called an \textit{$\mathcal{O}$-order} if $A$ is finitely generated projective as an $\mathcal{O}$-module. An $\mathcal{O}$-order $A$ is \textit{symmetric} if $\mathrm{Hom}_\mathcal{O}(A,\mathcal{O})$ is isomorphic to $A$ as $(A,A)$-bimodules. A finitely generated right $A$-module $M$ is called an $A$-\textit{lattice} if it is finitely generated projective as an $\mathcal{O}$-module\footnote[1]{In this paper, we consider $A=\mathcal{O}[X,Y]/(X^2,Y^2)$, which is a finitely generated Cohen--Macaulay $\mathcal{O}$-algebra with $\mathrm{Kr\text{-}dim}(A)=\mathrm{Kr\text{-}dim}(\mathcal{O})$. Thus, it follows from \cite[(1.8)]{Y} that a finitely generated $A$-module $M$ is a Cohen--Macaulay $A$-module if and only if it is a Cohen--Macaulay $\mathcal{O}$-module. Since $\mathcal{O}$ is regular,   ``$A$-lattices'' coincide with ``maximal Cohen--Macaulay $A$-modules'', see \cite[(1.5.1)]{Y}.}. 
We denote by $\mathsf{mod}$-$A$ the category consisting of finitely generated right $A$-modules and by $\mathsf{latt}$-$A$ the full subcategory of $\mathsf{mod}$-$A$ consisting of $A$-lattices. 

Let $A$ be a symmetric $\mathcal{O}$-order and $M$ a non-projective indecomposable $A$-lattice. Almost split sequences for $\mathsf{latt}$-$A$ had been studied by Auslander and Reiten. According to \cite{AR1},  there exists an almost split sequence ending at $M$ if and only if $M$ satisfies the following condition $(\natural)$:
$$
\text{$M\otimes_\mathcal{O}\mathcal{K}$ is projective as an $A\otimes_\mathcal{O}\mathcal{K}$-module. } \eqno(\natural)
$$
An almost split sequence ending at $M$ is unique up to isomorphism of short exact sequences if it exists. Therefore, we adopt the definition of the stable Auslander--Reiten quiver for $\mathsf{latt}$-$A$ as a valued quiver whose vertices are the isoclasses of non-projective indecomposable $A$-lattices satisfying $(\natural)$ in which there are valued arrows whenever there exists an irreducible morphism (Definition 1.11).
Unfortunately, it is too difficult to determine the stable Auslander--Reiten quiver for $\latt$-$A$ completely. Hence, we focus on a special kind of $A$-lattices called \textit{Heller lattices}, which are $A$-lattices defined as the direct summands of the first syzygies of indecomposable $A\otimes_{\mathcal{O}}\kappa$-modules viewed as $A$-modules. Note that Heller lattices satisfy the condition $(\natural)$. 
In this paper, we call a stable component containing indecomposable Heller lattices a \textit{Heller component} of $A$, and we denote by $\mathcal{CH}_{A}$ the union of Heller components of $A$. 
Some known results for determining $\mathcal{CH}_{A}$ are found in \cite{K2} and \cite{AKM}. In \cite{K2}, Kawata considered group algebras over $\mathcal{O}$ of characteristic zero with some assumption on ramification, and Ariki, Kase and the author considered truncated polynomial rings over $\mathcal{O}$ \cite{AKM}. 
By the definition of the Heller lattice, if a complete list of isoclasses of indecomposable modules over $A\otimes_\mathcal{O}\kappa$ is given, then we can determine $\mathcal{CH}_{A}$.  
Since non-projective-injective  indecomposable modules over a Brauer graph algebra (aka a symmetric special biserial algebra \cite{S}) are classified by using string paths and band paths  (see \cite{WW}, \cite{BR} or Subsection 1.3.),  it is natural to consider the case when $A\otimes_{\mathcal{O}}\kappa$ is a Brauer graph algebra. 

In this paper, we determine non-periodic components contained in $\mathcal{CH}_{A}$ of $A=\mathcal{O}[X,Y]/(X^2,Y^2)$. Note that $A\otimes_\mathcal{O}\kappa$ is a Brauer graph algebra associated with one loop and one vertex with multiplicity $1$. The main result is the following:

\begin{maintheorem} Let $\mathcal{O}$ be a complete discrete valuation ring, and $A=\mathcal{O}[X,Y]/(X^2,Y^2)$. Assume that the residue field of $\mathcal{O}$ is algebraically closed. For a string path $w$, let $M_w$ be the indecomposable $A\otimes_\mathcal{O}\kappa$-module given by $w$ and $Z_{M_w}$ the first syzygy of $M_w$ in $\latt$-$A$. Then, the following statements hold.
\begin{enumerate}
\item  If $w$ has even length, then $Z_{M_w}$ is indecomposable.
\item The Heller component $\mathcal{CH}_{A}$ contains a unique non-periodic component $\mathcal{CH}_\text{np}$.
\item  An indecomposable Heller lattice $Z$ lies on $\mathcal{CH}_\text{np}$ if and only if $Z=Z_{M_w}$ for some $w$ with even length.
\item $Z_{M_w}$ appears on the boundary of the component $\mathcal{CH}_\text{np}$.
\item The component $\mathcal{CH}_\text{np}$ is isomorphic to $\mathbb{Z}A_{\infty}$. 
\end{enumerate}
\end{maintheorem}

We note that the ``Kronecker algebra" over a ring $R$ usually means the generalized triangular matrix $R$-algebra 
$$
\left(\begin{array}{cc}
R& 0 \\
R^2 & R\end{array}\right).
$$
However, in this paper, we call the $R$-algebra $R[X,Y]/(X^2,Y^2)$ the ``Kronecker algebra'' following Erdmann, see \cite[Chapter I, Example 4.3]{Erd}. These two algebras are not isomorphic each other, but there is a functorial relation, which is explained in \cite[Section 5]{Gab}, \cite[X.2]{ARS} and \cite[Chapter XIX, 1.13 Remark]{SS1}.

This paper consists of five sections. In Section 1, we define almost split sequences and the stable Auslander--Reiten quiver for $\mathsf{latt}$-$A$, and recall some results from \cite{A1}, \cite{AKM}, \cite{Ri} and \cite{Z}.
In Section 2, we give a complete list of Heller lattices of $A=\mathcal{O}[X,Y]/(X^2, Y^2)$, and explain their properties including the indecomposability, the periodicity/aperiodicity and the appearance of non-periodic Heller lattices on the boundary of $\mathcal{CH}_\text{np}$. 
Moreover, we show that if the tree class of $\mathcal{CH}_\text{np}$ is not $A_{\infty}$, then the possibilities of the tree class are  $\widetilde{E}_{6}$, $\widetilde{E}_{7}$, $\widetilde{E}_{8}$, $\widetilde{F}_{41}$ or $\widetilde{F}_{42}$. 
In Section 3, we define an additive function on $\mathcal{CH}_\text{np}$ and we show that the tree class of $\mathcal{CH}_\text{np}$ is neither $\widetilde{F}_{41}$ nor $\widetilde{F}_{42}$. 
In Section 4, we prove the main result by computing the ranks of vertices of the component in $\Z\widetilde{E}_{6}$, $\Z\widetilde{E}_{7}$ or $\Z\widetilde{E}_{8}$ to exclude the cases. 
In the last section, we improve \cite[Theorem 1.27]{AKM}  as follows.

\begin{theo}
Let $A$ be a symmetric $\O$-order, where $\mathcal{O}$ is a complete discrete valuation ring, and let $\mathcal{C}$ be a component of the stable Auslander--Reiten quiver for $\latt$-$A$. Assume that $\mathcal{C}$ has infinitely many vertices. Then, the following statements hold.
\begin{enumerate}[(1)]
\item Suppose that $\mathcal{C}$ is $\tau$-periodic. Then, one of the following statements holds:
\begin{enumerate}[(i)]
\item If $\mathcal{C}$ has no loops, then $\mathcal{C}$ is of the form $\mathbb{Z}T/G$, where $T$ is a directed tree whose underlying graph is one of infinite Dynkin diagrams.
\item If $\mathcal{C}$ has loops, then $\mathcal{C}\setminus\{\text{loops}\}=\mathbb{Z}A_{\infty}/\langle \tau \rangle$. Moreover, the loops appear on the boundary of $\mathcal{C}$.
\end{enumerate}
\item Suppose that $\mathcal{C}$ is $\tau$-non-periodic. Then, $\mathcal{C}$ has no loops.  Moreover, if either
\begin{enumerate}[(i)]
\item $\mathcal{C}$ does not contain Heller lattices or
\item $A\otimes_{\O}\kappa$ has finite representation type, 
\end{enumerate}
then the tree class of $\mathcal{C}$ is one of infinite Dynkin diagrams or Euclidean diagrams.
\end{enumerate}
\end{theo}

Note that there may exist loops in Auslander--Reiten quivers by \cite{W}. 

\section*{Acknowledgment}
My heartfelt appreciation goes to Professor Susumu Ariki (Osaka University) who provided helpful comments and suggestions. I would also like to thank Professor Shigeto Kawata (Nagoya City University), Professor Michihisa Wakui (Kansai University), Professor Ryoichi Kase (Okayama University of Science), Professor Joseph Chuang (City University of London), Professor Steffen Koenig (Universit\"at Stuttgart) and Professor Liron Speyer (University of Virginia) whose meticulous comments were an enormous help to me.

\section{Preliminaries} 

Throughout this paper, we use the following conventions.
\begin{enumerate}[(a)]
\item $\mathcal{O}$ denotes a complete discrete valuation ring, $\kappa$ is the residue field and $\mathcal{K}$ is the quotient field. We assume that the residue field $\kappa$ is algebraically closed.
\item ``Modules" mean right modules.
\item Given an $\mathcal{O}$-order $A$, we write $\mathsf{latt}$-$A$ for the category of $A$-lattices. Given a pair of $A$-lattices $M$ and $N$, we denote by $\Hom _A(M,N)$ the $\mathcal{O}$-module of all $A$-homomorphisms from $M$ to $N$. 
\item Tensor products are taken over $\mathcal{O}$. 
\item For an $\mathcal{O}$-order $A$, we denote by $\mathsf{latt}^{(\natural)}$-$A$ the full subcategory of $\mathsf{latt}$-$A$ consisting of $A$-lattices $M$ such that $M\otimes\mathcal{K}$ is projective as an $A\otimes\mathcal{K}$-module. 
\item The symbol $\delta_{i,j}$ means the Kronecker delta.
\item The identity matrix of size $n$ is denoted by $I_n$.
\end{enumerate}
\subsection{Almost split sequences} 

In order to introduce the stable Auslander--Reiten quivers for $\mathsf{latt}^{(\natural)}$-$A$, we recall irreducible, minimal, and almost split morphisms. Main references for details are \cite{A1} and \cite{AKM}. Let $\mathscr{A}$ be an abelian category with enough projectives and  $\mathscr{C}$ an additive full subcategory closed under extensions and direct summands. Let $f:L\to M$  be a morphism in $\mathscr{C}$. The morphism $f$ is called \textit{left minimal} if every  $h\in\End _{\mathscr{C}}(M)$ with $hf=f$ is an isomorphism, and is called \textit{left almost split} if it is not a section and every $h\in\Hom_{\mathscr{C}}(L,W)$ which is not a section factors through $f$. 
Dually, a morphism $g:M\to N$ in $\mathscr{C}$ is called \textit{right minimal} if every $h\in\End _{\mathscr{C}}(M)$ with $gh=g$ is an isomorphism, and is called \textit{right almost split} if it is not a retraction and every $h\in\Hom_{\mathscr{C}}(W,N)$ which is not a retraction factors through $g$.
A morphism $f$ is said to be \textit{left minimal almost split}  in $\mathscr{C}$ if $f$ is both left minimal and left almost split. Similarly, a \textit{right minimal almost split morphism} in $\mathscr{C}$ is defined.

\begin{proposition}[{\cite[Proposition 4.4]{A1}}]\label{seq} Let $L$, $M$ and $N$ be objects of $\mathscr{C}$. The following statements are equivalent for a short exact sequence
\[ 0 \longrightarrow L \xrightarrow{\quad g\quad } M\xrightarrow{\quad f\quad} N\longrightarrow 0. \]
\begin{enumerate}[(1)]
\item  $f$ is right almost split in $\mathscr{C}$,  and $g$ is left almost split in $\mathscr{C}$.
\item  $f$ is minimal right almost split in $\mathscr{C}$.
\item  $f$ is right almost split and $\End_{\mathscr{C}}L$ is local.
\item  $g$ is minimal left almost split in $\mathscr{C}$.
\item  $g$ is left almost split in $\mathscr{C}$ and $\End_{\mathscr{C}}N$ is local.
\end{enumerate}\end{proposition}


\begin{definition}\label{almost} Let $L$, $M$ and $N$ be objects of $\mathsf{latt}^{(\natural)}$-$A$. A short exact sequence in $\mathsf{latt}^{(\natural)}$-$A$
\[ 0 \longrightarrow L \longrightarrow M\xrightarrow{\quad p\quad} N\longrightarrow 0\]
is called an \textit{almost split sequence ending at $N$} if the following two conditions are satisfied:
\begin{enumerate}[(i)]
\item The morphism $p$ is right almost split in $\mathsf{latt}^{(\natural)}$-$A$. 
\item The $A$-lattice $L$ is indecomposable.
\end{enumerate}
\end{definition}

 Let $\mathbb{E}: 0\to L \to E \to M \to 0$ be an almost split sequence in $\mathsf{latt}^{(\natural)}$-$A$. Then, it follows from Proposition \ref{seq} that any almost split sequence ending at $M$ is isomorphic to $\mathbb{E}$ as short exact sequences. Similarly, any almost split sequence starting from $L$ is isomorphic to $\mathbb{E}$ as short exact sequences. We denote by $\mathscr{E}(M)$ the almost split sequence ending at $M$. Here, we set $\tau(M)=L$ and $\tau^{-1}(L)=(M)$, and we call both $\tau$ and $\tau^{-1}$ \textit{AR translations}.

\begin{definition} Let $M$ and $N$ be objects in $\latt ^{(\natural)}$-$A$. A morphism $f\in\Hom_A(M,N)$ is said to be an \textit{irreducible morphism}, provided that
\begin{enumerate}[(i)]
\item the morphism $f$ is neither a section nor a retraction,
\item if $f=f_2\circ f_1$ in $\mathsf{latt}^{(\natural)}$-$A$, then either $f_{1}$ is a section or $f_{2}$ is a retraction.
\end{enumerate}
\end{definition}

It is well-known that almost split sequences are characterized by irreducible morphisms. The arguments in \cite[V.5, Proposition 5.9]{ARS} work without change in our setting. Note in particular that \cite[V.5, Theorem 5.3]{ARS} also holds in our setting.

\begin{lemma}[{\cite[V.5, Proposition 5.9]{ARS}}] Let $M$ be an $A$-lattice in $\mathsf{latt}^{(\natural)}$-$A$. Then, a short exact sequence in $\mathsf{latt}^{(\natural)}$-$A$ 
\[ 0\longrightarrow L \xrightarrow{\quad f \quad} E \xrightarrow{\quad g\quad} M \longrightarrow 0\]
is isomorphic to $\mathscr{E}(M)$ if and only if the morphisms $f$ and $g$ are irreducible.
\end{lemma}

\begin{proposition}[{\cite[Proposition 1.15]{AKM}}]\label{AKM}
Let $A$ be a symmetric $\mathcal{O}$-order, $M$ an indecomposable $A$-lattice in $\mathsf{latt}^{(\natural)}$-$A$, and let $p:P\to M$ be the projective cover of $M$ and $\Omega_{A} (M)$ the first syzygy of $M$, which lies in $\mathsf{latt}^{(\natural)}$-$A$. Given an endomorphism $\varphi : M \to M$, we obtain the pullback diagram along $p$ and $\varphi$:
$$\begin{xy}
(0,15)*[o]+{0}="01",(20,15)*[o]+{\Omega_{A}(M)}="L",(40,15)*[o]+{E}="E", (60,15)*[o]+{M}="M",(80,15)*[o]+{0}="02",
(0,0)*[o]+{0}="03",(20,0)*[o]+{\Omega_{A}(M)}="L2",(40,0)*[o]+{P}="nP", (60,0)*[o]+{M}="nM",(80,0)*[o]+{0}="04",
\ar "01";"L"
\ar "L";"E"
\ar "E";"M"
\ar "M";"02"
\ar "03";"L2"
\ar "L2";"nP"
\ar "nP";"nM"_{p}
\ar "nM";"04"
\ar @{-}@<0.5mm>"L";"L2"
\ar @{-}@<-0.5mm>"L";"L2"
\ar "E";"nP"
\ar "M";"nM"^{\varphi}
\end{xy}$$
Then, the following statements are equivalent.
\begin{enumerate}[(1)]
\item The upper short exact sequence is isomorphic to $\mathscr{E}(M)$.
\item The following three conditions hold.
\begin{enumerate}[(i)]
\item The morphism $\varphi$ does not factor through $p$.
\item $\Omega_{A}(M)$ is an indecomposable $A$-lattice.
\item For all $f\in \mathrm{rad}\End _A(M)$, the morphism $\varphi \circ f$ factors through $p$.
\end{enumerate} 
\end{enumerate}
In particular, we have an isomorphism $\tau(M)\simeq \Omega_{A}(M)$.
\end{proposition}

\subsection{Stable Auslander--Reiten quivers} 

In this subsection, we introduce the stable Auslander--Reiten quiver for $\mathsf{latt}^{(\natural)}$-$A$. We follow the notation of \cite{Z}.

Given a quiver $Q$, we denote by $Q_{0}$ and $Q_1$ the set of vertices and arrows, respectively. A pair $(Q,v)$ of a quiver $Q$ and a map $v:Q_{1}\rightarrow \Z_{\geq 0}\times \Z_{\geq 0}$ is called a \textit{valued quiver}, and the values of the map $v$ are called \textit{valuations}. For an arrow $x\to y$ of $Q$, we write $v(x\to y)=(d_{xy}, d_{xy}')$, and if there is no arrow from $x$ to $y$, we understand that $d_{xy}=d_{xy}'=0$. If $v(x\to y)=(1,1)$ for all arrows $x\to y$ of $Q$, then $v$ is said to be \textit{trivial}. For each vertex $x\in Q_0$, we set
\[ x^{+}=\{y\in Q_0\ |\ x\to y \in Q_1\},\quad x^{-}=\{y\in Q_0\ |\  y\to x\in Q_1\}. \]
A quiver $Q$ is \textit{locally finite} if $x^{+}\cup x^{-}$ is a finite set for any $x\in Q_0$. A \textit{stable translation quiver} is a pair $(Q,\tau)$ of a locally finite quiver $Q$ without multiple arrows and a quiver automorphism $\tau$ satisfying $x^{-}=(\tau x)^{+}$. Let $\mathcal{C}$ be a full subquiver of a stable translation quiver $(Q,\tau)$. Then, $\mathcal{C}$ is a (connected) \textit{component} if the following three conditions are satisfied.
\begin{enumerate}[(i)]
\item $\mathcal{C}$ is stable under the quiver automorphism $\tau$.
\item $\mathcal{C}$ is a disjoint union of connected components of the underlying undirected graph.
\item  There is no proper subquiver of $\mathcal{C}$ that satisfies (i) and (ii).
\end{enumerate}
In particular, $(Q,\tau)$ is \textit{connected} if $Q$ satisfies the above three conditions.

\begin{remark} In standard textbooks, loops are not allowed when we define a stable translation quiver, for example \cite{B}. However, we note that the definition of a stable translation quiver in \cite{Z} admits loops, and we adopt this definition. \end{remark}  

A \textit{valued stable translation quiver} is a triple $(Q, v, \tau)$ such that
\begin{enumerate}[(i)]
\item $(Q,v)$ is a valued quiver, 
\item $(Q,\tau)$ is a stable translation quiver,
\item  $v(\tau y\to x)=(d_{xy}',d_{xy})$ for each arrow $x\to y$.
\end{enumerate}

A group $G\subset \mathrm{Aut}((Q,v,\tau))$ is said to be \textit{admissible} if each $G$-orbit intersects $x^{+}\cup \{x\}$ in at most one vertex and $\{x\}\cup x^{-}$ in at most one vertex, for any $x\in Q_0$. For an admissible group $G$, we may form the valued stable translation quiver $(Q/G,v_G,\tau _G)$ such that $Q/G$ is the $G$-orbit quiver with the induced map $v_G$ and translation $\tau_G$.  

Given a valued quiver $(\Delta, v)$, one can construct the valued stable translation quiver $(\Z\Delta,\tilde{v},\tau)$ as follows \cite{Ri}.
\begin{itemize}
\item  $(\Z\Delta)_0 = \Delta _0\times \Z$.
\item Draw arrows $(n,x)\to (n,y)$ and $(n-1,y)\to (n,x)$ whenever an arrow $x\to y$ exists in $\Delta$.
\item The valuations  of $\tilde{v}$ are defined by 
\[ \tilde{v}((n,x)\to (n,y))=(d_{xy},d_{xy}'),\quad  \tilde{v}((n-1,y)\to (n,x))=(d_{xy}',d_{xy}). \]
\item The translation $\tau$ is defined by $\tau((n,x))=(n-1,x)$.
\end{itemize}
We will write it simply $\Z\Delta$ when no confusion can arise. The valued stable translation quiver $\Z\Delta$ has no loops whenever $\Delta$ has no loops.

Let $(Q,v,\tau)$ be a connected valued stable translation quiver. A vertex $x$ of $Q$ is called \textit{periodic} if $x=\tau ^k x$ for some $k>0$. If there is a periodic vertex in $Q$, then all vertices of $Q$ are periodic. In this case, $(Q,v,\tau)$ is called \textit{periodic} \cite{HPR}. $(Q,v,\tau)$ is said to be \textit{smooth} if $v$ is trivial and $\sharp x^{+}=2$ for all $x\in Q_0$. 

\begin{definition}\label{def of subadditive} Let $(Q,v, \tau)$ be a valued stable translation quiver. A \textit{subadditive function} on $(Q,v, \tau)$ is a function $\ell$ from $Q_0$ to the set of non-negative integers  $\mathbb{Z}_{\geq 0}$ such that it satisfies
\[ \ell(x)+\ell(\tau x)\geq \sum _{y \in  x^{-}}d_{yx}\ell(y) \]
for all  $x\in Q_0$. A subadditive function $\ell$ is called \textit{additive} if the equality holds for all $x\in Q_0$. 
\end{definition}

\begin{theorem}[{\cite[p.653, 669]{Z}}]\label{Z} Let $(Q,v,\tau)$ be a non-periodic connected valued stable translation quiver which admits a non-zero subadditive function $\ell: Q_0 \to \mathbb{Z}_{\geq 0}$. Then, one of the following holds:
\begin{enumerate}[(i)]
\item $(Q,v,\tau)$ is smooth and $d$ is both additive and bounded.
\item $(Q,v,\tau)$ is of the form $\Z\Delta$ for some valued quiver $\Delta$. 
\end{enumerate} 
Moreover, if $Q$ has a cyclic path, then $(Q,v,\tau)$ is smooth and $\ell$ is additive.\end{theorem}

The following theorems are useful to describe stable translation quivers. The former is showed by Riedtmann \cite{Ri} and the latter is showed by Happel, Preiser and Ringel \cite{HPR}.

\begin{theorem}[Riedtmann]\label{Ried}
Let $(Q,\tau)$ be a stable translation quiver without loops and $\mathcal{C}$ a connected component of $(Q,\tau)$. Then, there exist a directed tree $T$ and an admissible 
group $G\subseteq \mathrm{Aut}(\Z T)$ such that $\mathcal{C}\simeq \Z T/G$ as stable translation quivers. Moreover, the underlying undirected tree $\overline{T}$ of $T$ is uniquely determined by $\mathcal{C}$, and the admissible group $G$ is unique up to conjugation in $\mathrm{Aut}(\Z T)$.
\end{theorem}
In Theorem \ref{Ried}, the underlying undirected tree $\overline{T}$ is called the \textit{tree class} of $\mathcal{C}$.  
If $\ell(\tau x)=\ell(x)$ and there are no loops in $Q$, then a subadditive function $\ell$ on $(Q,v,\tau)$ from Definition \ref{def of subadditive} restricts a function on the tree class $\overline{T}$ that satisfies
\[2\ell (x)\ge \sum_{y\to x \text{ in } T} d_{yx}\ell (y)+\sum_{x\to y \text{ in }T} d'_{xy}\ell (y),  \]
and it gives a positive semidefinite Cartan matrix.

\begin{theorem}[Happel, Preiser, Ringel]\label{tree class}
Let $(\Delta,v)$ be a connected valued quiver without loops and multiple arrows. If $\Delta$ admits a non-zero function $f:\Delta _0\to \mathbb{Q}_{\geq 0}$ that satisfies 
\[ 2f(x)\ge \sum_{y\in x^-} d_{yx}f(y)+\sum_{y\in x^+} d'_{xy}f(y) \quad \text{for}\quad x\in \Delta _0, \]
then the following statements hold.
\begin{enumerate}[(1)]
\item The underlying undirected graph $\overline{\Delta}$ is either a finite or infinite Dynkin diagram or a Euclidean diagram.
\item If the inequality is strict for some $x\in\Delta_0$, then $\overline{\Delta}$ is either a finite Dynkin diagram or $A_{\infty}$.
\item If the equality holds for all $x\in\Delta_0$, then $\overline{\Delta}$ is either an infinite Dynkin diagram or a Euclidean diagram.
\item If $f$ is unbounded, then $\overline{\Delta}$ is $A_{\infty}$.
\end{enumerate}
\end{theorem}

\begin{definition}\label{ARquiver}
\begin{enumerate}[(a)]
\item The \textit{stable Auslander--Reiten quiver} for $\mathsf{latt}^{(\natural)}$-$A$ is the valued quiver defined as follows:
\begin{itemize}
\item The set of vertices is a complete set of isoclasses of non-projective indecomposable $A$-lattices in $\latt^{(\natural)}$-$A$.
\item We draw a valued arrow $M\xrightarrow{(a,b)}N$ whenever there exist irreducible morphisms $M\to N$, where the valuation $(a,b)$ means:
\begin{enumerate}[(i)]
\item For a minimal right almost split morphism $f: E\to N$, $M$ appears $a$ times in $E$ as direct summands. 
\item For a minimal left almost split morphism $g: M\to E$, $N$ appears $b$ times in $E$ as direct summands.
\end{enumerate}
\end{itemize}
The stable Auslander--Reiten quiver for $\mathsf{latt}^{(\natural)}$-$A$, which we called the stable Auslander--Reiten quiver for $\latt$-$A$ in the introduction, is denoted by $\Gamma_s(A)$.
\item The union of components of $\Gamma_s(A)$ containing indecomposable Heller lattices is said to be the \textit{Heller component} of $A$, and denoted by $\mathcal{CH}_{A}$.
\end{enumerate}
\end{definition}
 
By the definition, we note that a component $\mathcal{C}$ of $\Gamma_s(A)$ does not have multiple arrows, and $\tau M$ exists for each vertex $M$ of $\mathcal{C}$ by the existence of almost split sequences $0\to \tau M\to E\to M\to 0$. Thus, the equation $M^{-}=(\tau M)^{+}$ holds and $(\mathcal{C},\tau)$ is a valued stable translation quiver. 
However, if $A$ is maximal or Morita equivalent to a Bass order, then the Auslander--Reiten quiver of $A$ has a loop \cite{W}.
Therefore, it is necessary to argue whether loops exist in the stable Auslander--Reiten quiver of $A$. 

First, we recall Miyata's theorem \cite[Theorem 1]{Mi}.

\begin{theorem}[Miyata]\label{Miyata}
Let $R$ be a commutative noetherian ring and $\Lambda$ an $R$-algebra which is of finite type as an $R$-module. 
Let $\mathbb{E}:0\to L\to E\to M \to 0$ be  a short exact sequence in $\mathsf{mod}$-$\Lambda$. If $E\simeq L\oplus M$ as $\Lambda$-modules, then $\mathbb{E}$ splits.
\end{theorem}

\begin{lemma}\label{valuation_of_loops}
Let $A$ be a symmetric $\mathcal{O}$-order, $\mathcal{C}$ a component of $\Gamma_s(A)$.
If a vertex $M\in\mathcal{C}_0$ has a loop, then $M\simeq \tau M$ and the valuation of the loop is $(1,1)$.
\end{lemma}
\begin{proof} Let $M$ be a vertex of $\mathcal{C}$. Suppose that $M\not\simeq \tau M$. Then, $\mathscr{E}(M)$ is of the form
\[ 0\longrightarrow \tau M\longrightarrow M^{\oplus l_1}\oplus \tau M^{\oplus l_2}\oplus E \longrightarrow M \longrightarrow 0, \]
where $l_1,$ $l_2\geq 1$.
Thus, we have 
\[ (l_1-1)\mathrm{dim}_{\kappa}(M\otimes\kappa)+(l_2-1)\mathrm{dim}_{\kappa}(\tau M\otimes\kappa) +\mathrm{dim}_{\kappa}(E\otimes\kappa)=0, \]
and hence $l_1=l_2=1$ and $E=0$. In this case, the short exact sequence $\mathscr{E}(M)$ splits by Theorem \ref{Miyata}, a contradiction.

Suppose that $M\simeq \tau M$. Then, $\mathscr{E}(M)$ is of the form
\[ 0\longrightarrow M\longrightarrow M^{\oplus l}\oplus E \longrightarrow M \longrightarrow 0, \]
where $l\geq 1$.
Thus, we have 
\[ (l-2)\mathrm{dim}_{\kappa}(M\otimes\kappa)+\mathrm{dim}_{\kappa}(E\otimes\kappa)=0, \]
and hence $l\leq 2$. If $l=2$, then $E=0$. In this case, the short exact sequence $\mathscr{E}(M)$ splits by Theorem \ref{Miyata}, a contradiction. Thus, $l=1$.
\end{proof}

\subsection{Indecomposable modules over a special biserial algebra}

Throughout this subsection, $\Lambda$ is a basic finite dimensional algebra over an algebraically closed field $\k$. Then, there exist a quiver $Q$ and an admissible ideal $\mathcal{I}$ in the path algebra $\k Q$ such that $\Lambda$ is isomorphic to the bound quiver algebra $\k Q/\mathcal{I}$. Moreover, there is a $\k$-linear equivalence between $\mathsf{mod}$-$\Lambda$ and $\mathsf{rep}(Q,\mathcal{I})$, where $\mathsf{rep}(Q,\mathcal{I})$ is the category of finite dimensional $\k$-linear representations of $\k Q/\mathcal{I}$, see \cite[Chapters II and III]{ASS}. We identify these two categories. 
\begin{definition} An algebra $\Lambda\simeq \k Q/\mathcal{I}$ is called \textit{special biserial} if the following two conditions are satisfied.
\begin{enumerate}[(i)]
\item For each vertex $x$ of $Q$, $\sharp x^{+}\leq 2$ and $\sharp x^{-}\leq 2$.
\item For each arrow $\alpha$ of $Q$, there exist at most one arrow $\beta$ such that $\alpha\beta\notin\mathcal{I}$ and at most one arrow $\gamma$ such that $\gamma\alpha\notin\mathcal{I}$.
\end{enumerate}
\end{definition}
Brauer graph algebras are symmetric special biserial algebras. The converse is also true by Schroll \cite{S}. Wald and Waschb\"{u}sch showed that special biserial algebras are of tame representation type by classifying indecomposable modules over a special biserial algebra into ``string modules'' and ``band modules'' \cite{WW}. Moreover, we can construct all indecomposable modules over a special biserial algebra by using a combinatorial method. In this subsection, we recall the construction of indecomposable modules over a special biserial algebra, see \cite{Erd}, \cite{HL} for details. 

\para \textbf{Strings and bands.}
Let $Q$ be a quiver. For an arrow $\alpha\in Q_1$, we denote by $s(\alpha)$ and $t(\alpha)$ the source of $\alpha$ and the target of $\alpha$, respectively.
Set $Q_{1}^{\ast}=\{ \alpha^{\ast}\ |\ \alpha\in Q_{1}\}.$ We understand that the symbol $\alpha^{\ast}$ is the formal inverse arrow of $\alpha$, that is, $\alpha^\ast$ is an arrow such that $s(\alpha^{\ast})=t(\alpha)$, $t(\alpha^{\ast})=s(\alpha)$ and $\alpha^{\ast\ast}=\alpha$. 
For a path $w=c_1c_2\cdots c_n$ in $Q$, we define $s(w)=s(c_1)$, $t(w)=t(c_n)$ and $w^{\ast}=c_n^{\ast}c_{n-1}^{\ast}\cdots c_1^{\ast}$. If $w$ is the path with the length $0$ at a vertex $a$, then we understand that $w$ is the trivial path $\varepsilon _a$ with $s(\varepsilon _a)=t(\varepsilon _a)=a$ and $\varepsilon _a^{\ast}=\varepsilon _a$. A \textit{walk} with length $n$ is a sequence $w=c_1c_2\cdots c_n$ such that each $c_i\in Q_1\cup Q_1^{\ast}$ and $t(c_i)=s(c_{i+1})$ for $i=1,2,\ldots , n-1$, and $w$ is called \textit{reduced} if $w$ is either a trivial path or a walk with positive length such that $c_{i+1}\neq c_{i}^{\ast}$ for all $i=1,2,\ldots ,n-1$. Given a walk $w$, the source $s(w)$ and the target $t(w)$ are also defined.
For two walks $w_1=c_{11}\cdots c_{1n} $ and $w_2=c_{21}\cdots c_{2m}$, the product $w_1w_2$ is defined by
\[ w_1w_2:=c_{11}\cdots c_{1n}c_{21}\cdots c_{2m} \]
when $t(w_1)=s(w_2)$. If $w$ is a walk with $s(w)=t(w)$, then one has also arbitrary powers $w^j$ of $w$.
Assume that $w=c_1c_2\cdots c_n$ is a reduced walk with positive length. The walk $w$ is called a \textit{reduced cycle} if $s(w)=t(w)$ and $c_n\neq c_1^{\ast}$. We say that a non-trivial path $p$ is \textit{contained} in $w$ if $p$ or $p^{\ast}$ is a subwalk of $w$.

A path $w$ is called a \textit{zero path} if $w$ belongs to $\mathcal{I}$. A zero path with minimal length is called a \textit{zero relation} of $\Lambda$. Let $p$ and $q$ be non-zero paths from a vertex $a$ to a vertex $b$. If $\lambda p+\mu q\in\mathcal{I}$ for some $\lambda\neq 0$ and $\mu\neq 0$, then the pair $(p,q)$ is called a \textit{binomial relation} of $\Lambda$.

\begin{definition}\label{string} A reduced walk $w$ is said to be a \textit{string path} of $\Lambda$ if each path contained in $w$ is neither a zero relation nor a maximal subpath of a binomial relation of $\Lambda$.
\end{definition}

\begin{definition}  A non-trivial reduced cycle is said to be a \textit{band path} of $\Lambda$ if each of its powers is a string path and it is not a power of a string path with less length.
\end{definition}

\para \textbf{String modules.}
For each string path $w$ of $\Lambda$, the \textit{string module} $M(w)$ is defined as follows. If $w=\varepsilon_a$, then $M(w)$ is the simple $\Lambda$-module corresponding to $a$. For a non-trivial $w=c_1c_2\cdots c_n$, $M(w)$ is the $\k$-linear representation $(M(w)_a, M(w)_\alpha)$ given by the following. For $1\leq i\leq n+1$, we set $\k(i)=\k$. Given a vertex $a$ of $Q$, we define $M(w)_a=\bigoplus _{i\in \mathcal{W}_a}\k(i)$, where 
\[ \mathcal{W}_a=\{ i \ |\ s(c_i)=a\}\cup \{ n+1\ |\ t(c_n)=a\}. \]
For $1\leq i\leq n$, we define the $\k$-linear map $f_{c_i}$ by
\[ f_{c_i}:\left\{\begin{array}{lll}
\k(i) \longrightarrow \k(i+1), & x\longmapsto x &\text{if $c_i\in Q_1$}, \\
\k(i+1) \longrightarrow \k(i), & x\longmapsto x &\text{if $c_i\in Q_1^\ast$}.
\end{array}\right. \]
Given an arrow $\alpha$ of $Q$, we define $M(w)_\alpha$ as the direct sum of the $\k$-linear maps $f_{c_i}$ such that $c_i=\alpha$ or $c_i^{\ast}=\alpha$. 

\para \textbf{Band modules.}
Let $w=c_1c_2\cdots c_n$ be a band path of $\Lambda$ and $V$ a finite dimensional indecomposable left $\k[x,x^{-1}]$-module. 
We construct the \textit{band module} $N(w, V)$ corresponding to $w$ and $V$ as follows. For $1\leq i\leq n$, we set $V(i)=V$. For $1\leq i \leq n$, let $f_{c_i}'$ be the $\k$-linear map defined by
\[ f_{c_i}':\left\{\begin{array}{lll}
V(i) \longrightarrow V(i+1), & v\longmapsto v &\text{if $1\leq i\leq n-1$ and $c_i\in Q_1$}, \\
V(i+1) \longrightarrow V(i), & v\longmapsto v &\text{if $1\leq i\leq n-1$ and $c_i\in Q_1^\ast$},\\
V(n) \longrightarrow V(1), & v\longmapsto xv &\text{if $i=n$ and $c_n\in Q_1$}, \\
V(1) \longrightarrow V(n), & v\longmapsto x^{-1}v &\text{if $i=n$ and $c_n\in Q_1^\ast$}.
\end{array}\right. \]
For a vertex $a$ of $Q$, we define $N(w,V)_a=\bigoplus _{i\in \mathcal{W}_a'}V(i)$, where 
\[ \mathcal{W}_a'=\{ i \ |\ s(c_i)=a\}. \]
For an arrow $\alpha$ of $Q$, we define $N(w, V)_\alpha$ as the direct sum of the $\k$-linear maps $f_{c_i}'$ such that $c_i=\alpha$ or $c_i^{\ast}=\alpha$. 

\begin{theorem}[{\cite[(2.3) Proposition]{WW}}]\label{stringband}
Let $\Lambda$ be a special biserial algebra. 
Then, the disjoint union of string modules, band modules and all  projective-injective modules corresponding to the binomial relations forms a complete set of isoclasses of finite dimensional  indecomposable $\Lambda$-modules.
\end{theorem} 

\begin{remark}
\begin{enumerate}
\item Let $w_1$ and $w_2$ be string paths of $\Lambda$. Then, the string modules $M(w_1)$ and $M(w_2)$ are isomorphic each other if and only if $w_2=w_1$ or $w_2=w_1^\ast$.
\item Let $w=c_1\cdots c_n$ be a band path. A \textit{rotation} of $w$ is a walk of the form $c_{i+1}\cdots c_n c_1\cdots c_{i}$. Given two band paths $w_1$ and $w_2$, the band modules $N(w_1,V)$ and $N(w_2,V)$ are isomorphic each other if and only if $w_2$ is a rotation of $w_1$ or a rotation of $w_1^{\ast}$.
\item A finite dimensional left $\k[x,x^{-1}]$-module is a finite dimensional $\k$-vector space together with a $\k$-linear automorphism $f$. If the module is indecomposable, then $f$ is similar to a Jordan block 
\[ J(\lambda,m) := \left(\begin{array}{ccccc}
\lambda & 1             & \cdots & \cdots & 0 \\
0            & \lambda  &  \cdots          & \cdots & 0 \\
\vdots    &                 & \ddots  &  & \vdots  \\
0           &     \cdots   &      \cdots      & \lambda & 1 \\
0           &     \cdots   &  \cdots &  0      & \lambda 
\end{array}\right) \]
for some $\lambda\in\k^{\times}$ and the size $m\in\mathbb{Z}_{>0}$. 
\end{enumerate}
\end{remark}

\section{The Kronecker algebra and almost split sequences} 

The main aim of this section is to present a complete list of isoclasses of indecomposable Heller lattices over the Kronecker algebra $A=\mathcal{O}[X,Y]/(X^{2},Y^{2})$, and compute almost split sequences ending at non-periodic indecomposable Heller lattices. From this section to the end of this paper, we set $A=\mathcal{O}[X,Y]/(X^{2},Y^{2})$.
For a positive integer $k$, we denote by $\{e_l\}_{l=1,2,\ldots,k}$ the canonical $\mathcal{O}$-basis of $\mathcal{O}^{\oplus k}$. Then an $\mathcal{O}$-basis of the direct sums of $k$ copies of $A$ is given by $\{e_l, Xe_l, Ye_l, XYe_l\}_{l=1,2,\ldots,k}$. 
Since $A\otimes \kappa$ is the Brauer graph algebra associated with one loop and one vertex with multiplicity one, the algebra $A\otimes\kappa$ is a special biserial algebra, which is given by the quiver with one vertex and two loops $\beta_1,\beta_2$ bound by the relations $\beta_1^2=\beta_2^2=0$ and $\beta_1\beta_2-\beta_2\beta_1=0$, where $\beta_1=X\otimes 1$, $\beta_2=Y\otimes 1\in A\otimes\kappa$.

\subsection{Indecomposable modules and Heller lattices}\label{indec mod Heller}
In this subsection,  we give a complete list of isoclasses of Heller lattices over $A$, and explain some properties of non-periodic Heller lattices. 

For simplicity, we visualize an $A\otimes\kappa$-module as follows:
\begin{itemize}
\item vertices represent basis vectors of the underlying $\kappa$-vector spaces,
\item arrows of the form $\longrightarrow$ represent the action of $X$, and $\dashrightarrow$ represent the action of $Y$. 
\item If there is no arrow (resp. dotted arrow) starting at a vertex, then $X$ (resp. $Y$) annihilates the corresponding basis element.
\end{itemize}
For example, the unique indecomposable projective module $A\otimes\kappa$ is described as
$$
A\otimes\kappa=\kappa 1\oplus \kappa X \oplus \kappa Y \oplus \kappa XY =  
\begin{xy}
(0,0)*[o]+{1}="1",(15,7)*[o]+{X}="X",(15,-7)*[o]+{Y}="Y",(30,0)*[o]+{XY}="XY",
\ar @{-->}"1";"Y"_{Y}
\ar "1";"X"^{X}
\ar @{-->} "X";"XY"^{Y}
\ar "Y";"XY"_{X}
\end{xy}. 
$$

By using the construction of indecomposable modules which is explained in Subsection 1.3, we obtain all finite dimensional indecomposable modules over $A\otimes\kappa$. 
\begin{enumerate}[(i)]
\item The string module $M(m):=M((\beta_1^{\ast}\beta_2)^m)$ $(m\in\mathbb{Z}_{\geq 0})$ is given by the formula:
$$ M(m) =\left(\bigoplus _{i=1}^{m}\kappa u_i\right)\oplus\left(\bigoplus_{j=0}^{m}\kappa v_j\right) = \begin{xy}
(0,7)*[o]+{u_1}="u1",(20,12)*[o]+{v_0}="v0",(20,-7)*[o]+{v_{m-1}}="vp-1",(0,-7)*[o]+{u_{m-1}}="up-1",
(0,2)*[o]+{\vdots}="2",(0,-2)*[o]+{\vdots}="22",
(0,-12)*[o]+{u_m}="up",(20,7)*[o]+{v_1}="v1",(20,-12)*[o]+{v_m}="vp",
(20,2)*[o]+{\vdots}="11",(20,-2)*[o]+{\vdots}="111",
\ar @{-->}"up";"vp"
\ar "up";"vp-1"
\ar @{-->}"u1";"v1"
\ar @{-->}"up-1";"vp-1"
\ar "u1";"v0"
\end{xy}=\left\{\begin{array}{ll}
Xu_i=v_{i-1} & 1\leq i \leq m, \\
Yu_i=v_i & 1\leq i\leq m, \\
Xv_i=Yv_i=0 & 0\leq i\leq m.
\end{array}\right.$$

\item The string module $M(-m):=M((\beta_1\beta_2^\ast)^m)$ $(m\in\mathbb{Z}_{\geq 0})$ is given by the formula:
$$  M(-m) \hspace{-1mm}
=\hspace{-1mm}\left(\bigoplus _{i=1}^{m+1}\kappa u_i\right)\oplus\left(\bigoplus_{j=1}^{m}\kappa v_j\right) \hspace{-1mm}
=\hspace{-4mm} \begin{xy}
(0,7)*[o]+{u_1}="u1",
(20,-7)*[o]+{v_{m-1}}="vp-1",(0,-12)*[o]+{u_{m}}="up-1",
(0,-3)*[o]+{\vdots}="2",(0,2)*[o]+{u_2}="u2",(20,2)*[o]+{v_2}="v2",
(0,-17)*[o]+{u_{m+1}}="up",(20,-12)*[o]+{v_{m}}="vp",
(20,7)*[o]+{v_1}="v1",
(20,-2)*[o]+{\vdots}="11",
\ar @{-->}"up";"vp"
\ar @{-->}"up-1";"vp-1"
\ar @{-->}"u2";"v1"
\ar "u1";"v1"
\ar "u2";"v2"
\ar "up-1";"vp"
\end{xy}
\hspace{-3mm}=\left\{\begin{array}{ll}
Xu_i=v_{i} & 1\leq i \leq m, \\
Xu_{m+1}=0, & \\
Yu_i=v_{i-1} & 2\leq i\leq m+1, \\
Xv_i=Yv_i=0 & 1\leq i\leq m.
\end{array}\right.$$

\item The string module $M(0)_{n}:=M((\beta_1\beta_2^{\ast})^{n-1}\beta_1)$ $(n\in\mathbb{Z}_{\geq 1})$ is given by the formula:
$$  M(0)_{n} \hspace{-1mm}=\hspace{-1mm}\left(\bigoplus _{i=1}^{n}\kappa u_i\right)\oplus\left(\bigoplus_{j=1}^{n}\kappa v_j\right) \hspace{-1mm}=\hspace{-3mm}
 \begin{xy}
(0,7)*[o]+{u_1}="u1",
(20,-7)*[o]+{v_{n-1}}="vp-1",(0,-7)*[o]+{u_{n-1}}="up-1",
(0,-2)*[o]+{\vdots}="2",(0,2)*[o]+{u_2}="u2",(20,2)*[o]+{v_2}="v2",
(0,-12)*[o]+{u_{n}}="up",(20,-12)*[o]+{v_{n}}="vp",
(20,7)*[o]+{v_1}="v1",
(20,-2)*[o]+{\vdots}="11",
\ar "up";"vp"
\ar "up-1";"vp-1"
\ar @{-->}"u2";"v1"
\ar "u1";"v1"
\ar "u2";"v2"
\ar @{-->}"up";"vp-1"
\end{xy}
\hspace{-3mm}
=\left\{\begin{array}{ll}
Xu_i=v_{i} & 1\leq i \leq n, \\
Yu_1=0, &  \\
Yu_i=v_{i-1} & 2\leq i\leq n, \\
Xv_i=Yv_i=0 & 1\leq i\leq n.
\end{array}\right.$$

\item The string module $M(\infty)_{n}:=M(\beta_2(\beta_1^{\ast}\beta_2)^{n-1})$ $(n\in\mathbb{Z}_{\geq 1})$ is given by the formula:
$$ M(\infty)_{n}=\left(\bigoplus _{i=1}^{n}\kappa u_i\right)\oplus\left(\bigoplus_{j=1}^{n}\kappa v_j\right) \hspace{-1mm}=\hspace{-3mm} 
\begin{xy}
(0,12)*[o]+{u_1}="u0",
(0,7)*[o]+{u_2}="u1",(20,12)*[o]+{v_1}="v0",(20,-7)*[o]+{v_{n-1}}="vp-1",(0,-7)*[o]+{u_{n-1}}="up-1",
(0,2)*[o]+{\vdots}="2",(0,-2)*[o]+{\vdots}="22",
(0,-12)*[o]+{u_n}="up",(20,7)*[o]+{v_2}="v1",(20,-12)*[o]+{v_n}="vp",
(20,2)*[o]+{\vdots}="11",(20,-2)*[o]+{\vdots}="111",
\ar @{-->}"up";"vp"
\ar @{-->}"u0";"v0"
\ar "up";"vp-1"
\ar @{-->}"u1";"v1"
\ar @{-->}"up-1";"vp-1"
\ar "u1";"v0"
\end{xy} \hspace{-3mm}=\left\{\begin{array}{ll}
Xu_1=0, \\
Xu_i=v_{i-1} & 2\leq i \leq n, \\
Yu_i=v_i & 1\leq i\leq n, \\
Xv_i=Yv_i=0 & 1\leq i\leq n.
\end{array}\right.$$

\item Let $V$ be a finite dimensional indecomposable left $\kappa[x,x^{-1}]$-module. Assume that $V$ is represented by $x\mapsto J(\lambda, n)$ with respect to a basis of $V$ for some $\lambda\in\kappa^{\times}$ and $n\in\mathbb{Z}_{\geq 1}$. The band module $M(\lambda)_{n}:=N(\beta_2^{\ast}\beta_1, V)$ is given by the formula:
$$ M(\lambda)_{n}=\left(\bigoplus _{i=1}^{n}\kappa u_i\right)\oplus\left(\bigoplus_{j=1}^{n}\kappa v_j\right) =
\left\{\begin{array}{ll}
Xu_i=v_{i} & 1\leq i \leq n, \\
Yu_1=\lambda v_1,\\
Yu_i=\lambda v_i +v_{i-1} &  2\leq i \leq n,\\
Xv_i=Yv_i=0 & 1\leq i\leq n.
\end{array}\right.$$
\end{enumerate}

Throughout this paper, we adopt the $\kappa$-basis of an indecomposable $A\otimes\kappa$-module described above. 

\begin{lemma}\label{indecc}
The set of the $A\otimes\kappa$-modules 
\[ \{ M(m) \mid m\in\mathbb{Z}\}\sqcup\{ M(\lambda)_{n} \mid \lambda\in \mathbb{P}^{1}(\kappa),\ n\in\mathbb{Z}_{\geq 1}\} \sqcup \{A\otimes\kappa\}, \]
where $\mathbb{P}^1(\kappa)$ is the projective line of $\kappa$, forms a complete set of isoclasses of finite dimensional indecomposable modules over $A\otimes\kappa$.
\end{lemma}
\begin{proof}
The statement follows from Theorem \ref{stringband}.
\end{proof}

\begin{remark}\label{ARseq} Almost split sequences for $\mathsf{mod}$-$\kappa[X,Y]/(X^2,Y^2)$ are known to be as follows:
$$
\begin{array}{ll}
0\longrightarrow M(-1)\longrightarrow (A\otimes\kappa)\oplus M(0)^{\oplus 2}\longrightarrow M(1)\longrightarrow 0&\\
0\longrightarrow M(m-1) \longrightarrow M(m)\oplus M(m)\longrightarrow M(m+1)\longrightarrow 0 & \text{if $m\neq 0$}\\
0\longrightarrow M(\lambda)_{n} \longrightarrow  M(\lambda)_{n-1}\oplus  M(\lambda)_{n+1}\longrightarrow M(\lambda)_{n}\longrightarrow 0 & n\geq 1, \lambda\in \mathbb{P}^1(\kappa)
\end{array}
$$
Here, if $n=1$, then we understand that $M(\lambda)_0=0$.
\end{remark}

\para \textbf{Heller lattices.}\label{Heller lattices}
Let $M$ be a non-projective indecomposable $A\otimes\kappa$-module given in Lemma \ref{indecc}. We view $M$ as an $A$-module. Then, the projective cover of $M$ as an $A$-module is given by $\pi_M:A^{\oplus\sharp\{u_i\}}\to M,\ e_i\mapsto u_i$.
For $m\in\mathbb{Z}$, $n\in\mathbb{Z}_{\geq 1}$ and  $\lambda\in\mathbb{P}^1(\kappa)$, we define the Heller $A$-lattices $Z_n$ and $Z_{m}^{\lambda}$ to be the $A$-lattices 
\[ Z_m:=\mathrm{Ker}(\pi_{M(m)}),\quad Z_n^{\lambda}:=\mathrm{Ker}(\pi_{M(\lambda)_n}). \]
We denote by $\mathbb{B}(m)$ and $\mathbb{B}(\lambda)_n$ the following $\mathcal{O}$-basis of Heller lattices $Z_m$ and $Z_n^{\lambda}$, respectively:
For $m>0$, 
\begin{align*} 
Z_m= &\ \mathcal{O}\varepsilon e_1 \oplus \mathcal{O}\varepsilon Xe_1 \oplus \mathcal{O} (Ye_1-Xe_2) \oplus \mathcal{O} XYe_1 \\
& \oplus \mathcal{O}\varepsilon e_2 \oplus \mathcal{O}\varepsilon Xe_2 \oplus \mathcal{O} (Ye_2-Xe_3) \oplus \mathcal{O} XYe_2 \\
& \oplus \cdots \\
&  \oplus \mathcal{O}\varepsilon e_{m-1} \oplus \mathcal{O}\varepsilon Xe_{m-1} \oplus \mathcal{O} (Ye_{m-1}-Xe_m) \oplus \mathcal{O} XYe_{m-1} \\
& \oplus \mathcal{O}\varepsilon e_m \oplus \mathcal{O}\varepsilon Xe_m \oplus \mathcal{O} \varepsilon Ye_m\oplus \mathcal{O} XYe_m, \\
 Z_0= &\ \mathcal{O}\varepsilon e_1\oplus\mathcal{O} Xe_1\oplus\mathcal{O}Ye_1\oplus\mathcal{O}XYe_1, \\
Z_{-m}= &\ \mathcal{O}\varepsilon e_1 \oplus \mathcal{O}\varepsilon Xe_1 \oplus \mathcal{O} Ye_1\oplus \mathcal{O} XYe_1 \\
& \oplus \mathcal{O}\varepsilon e_2 \oplus \mathcal{O}\varepsilon Xe_2 \oplus \mathcal{O} (Ye_2-Xe_1) \oplus \mathcal{O} XYe_2 \\
& \oplus \cdots \\
& \oplus \mathcal{O}\varepsilon e_{m} \oplus \mathcal{O}\varepsilon Xe_{m} \oplus \mathcal{O} (Ye_{m}-Xe_{m-1})\oplus \mathcal{O} XYe_{m} \\
& \oplus \mathcal{O}\varepsilon e_{m+1} \oplus \mathcal{O}Xe_{m+1} \oplus \mathcal{O}  (Ye_{m+1}-Xe_{m})\oplus \mathcal{O} XYe_{m+1}.
\end{align*} 
For $n>1$,
\begin{align*}
Z_n^{\lambda} = &\ \mathcal{O}\varepsilon e_1 \oplus \mathcal{O} \varepsilon Xe_1 \oplus \mathcal{O} (Ye_1-\lambda Xe_1) \oplus \mathcal{O}XY e_1 \\
&\oplus \mathcal{O}\varepsilon e_2 \oplus \mathcal{O} \varepsilon Xe_2 \oplus \mathcal{O} (Ye_2-\lambda Xe_2-Xe_1) \oplus \mathcal{O}XY e_2 \\
&\oplus \cdots \\
&\oplus \mathcal{O}\varepsilon e_n \oplus \mathcal{O} \varepsilon Xe_n \oplus \mathcal{O}(Ye_n-\lambda Xe_n-Xe_{n-1}) \oplus \mathcal{O}XY e_n\\
Z_n^{\infty} = &\ \mathcal{O}\varepsilon e_1 \oplus \mathcal{O} Xe_1 \oplus \mathcal{O} (Ye_1-Xe_2) \oplus \mathcal{O}XY e_1 \\
&\oplus \mathcal{O}\varepsilon e_2 \oplus \mathcal{O} \varepsilon Xe_2 \oplus \mathcal{O} (Ye_2-Xe_3) \oplus \mathcal{O}XY e_2 \\
&\oplus \cdots  \\
&\oplus \mathcal{O}\varepsilon e_{n-1} \oplus \mathcal{O} \varepsilon Xe_{n-1} \oplus \mathcal{O} (Ye_{n-1}-Xe_n) \oplus \mathcal{O}XY e_{n-1} \\
&\oplus \mathcal{O}\varepsilon e_n \oplus \mathcal{O} \varepsilon Xe_n \oplus \mathcal{O} \varepsilon Ye_n \oplus \mathcal{O}XY e_n,
\end{align*}

and

\begin{align*}
Z_1^{\lambda} = &\ \mathcal{O}\varepsilon e_1 \oplus \mathcal{O} \varepsilon Xe_1 \oplus \mathcal{O} (Ye_1-\lambda Xe_1) \oplus \mathcal{O}XY e_1 ,\\
Z_1^{\infty} = &\ \mathcal{O}\varepsilon e_1 \oplus \mathcal{O} Xe_1 \oplus \mathcal{O} \varepsilon Ye_1 \oplus \mathcal{O}XY e_1 .
\end{align*}

From now on, we explain some properties of the Heller lattices. The main claim of this subsection is the following proposition.

\begin{proposition}\label{Heller}
For $m\in\mathbb{Z}$, $n\in \Z_{\geq 1}$ and $\lambda\in\mathbb{P}^1(\kappa)$, let $Z_m$, $Z_n^{\lambda}$ be $A$-lattices as above. Then the following statements hold.
\begin{enumerate}[(1)]
\item There are isomorphisms 
\[ Z_m\otimes \kappa\simeq M(m-1)\oplus M(m),\quad Z_n^{\lambda}\otimes\kappa\simeq M(\lambda)_n\oplus M(-\lambda)_n, \]
where we set $-\infty=\infty$.
\item The Heller lattice $Z_m$ is indecomposable.
\end{enumerate}
\end{proposition}

The proof of (1) in Proposition \ref{Heller} is straightforward. It follows from the statement (1) that the number of indecomposable direct summands of the Heller lattices described in \ref{Heller lattices} is at most two. Furthermore, the statement (1) also implies that the Heller lattices $Z_n$ and $Z_m$ are not isomorphic whenever $m\neq n$. We use the next lemma to prove the statement (2) in Proposition \ref{Heller}.

\begin{lemma}\label{divide} Let $Z$ be a Heller lattice over $A$. Then, the rank of $Z$ as an $\mathcal{O}$-module is divisible by four.
\end{lemma}
\begin{proof} Let $Z$ be a Heller $A$-lattice. Then, $Z\otimes\K$ is projective as an $A\otimes\K$-module. On the other hand, the unique projective indecomposable $A\otimes\K$-module is $A\otimes\K$, whose dimension is four. This gives the desired conclusion.\end{proof}

\para \textbf{Proof of (2) in Proposition \ref{Heller}.}
For an integer $m$, we obtained an isomorphism $Z_m\otimes \kappa\simeq M(m)\oplus M(m-1)$ by Proposition \ref{Heller} (1). Assume that $Z_m$ is decomposable. We write $Z_m=Z^1\oplus Z^2$ with $Z^i\neq 0\ (i=1,2)$. By the Krull--Schmidt--Azumaya theorem, we would obtain two isomorphisms $Z^1\otimes\kappa\simeq M(m)$ and $Z^2\otimes\kappa\simeq M(m-1)$. On the other hand,  the dimension of $M(m)$ as a $\kappa$-vector space is odd, a contradiction with Lemma \ref{divide}. Therefore, $Z_m$ is an indecomposable $A$-lattice, and we have completed the proof of (2) in Proposition \ref{Heller}.

\subsection{The non-periodic Heller component}

In this subsection, we show that the Heller component of $A=\mathcal{O}[X,Y]/(X^2,Y^2)$ contains a unique non-periodic component. We denote by $\mathcal{CH}_{\text{np}}$ the union of non-periodic components of $\mathcal{CH}_{A}$. 
The aim of this subsection is to show the following proposition.
\begin{proposition}\label{Heller2}
The following statements hold.
\begin{enumerate}
\item For any integer $m$, there exists an isomorphism $\tau(Z_{m})\simeq Z_{m-1}$. Thus, we obtain the following $\tau$-orbit:
$$\cdots\quad
\begin{xy}
(-20,0) *[o]+{Z_{-2}}="E",(-3,0) *[o]+{Z_{-1}}="A",(15,0)*[o]+{Z_{0}}="B",
(30,0)*[o]+{Z_1}="C",(45,0) *[o]+{Z_2}="D",(57,0) *{}="F",(-33,0) *{}="O",
\ar @{-->}_{\tau}"B";"A"
\ar @{-->}_{\tau}"C";"B"
\ar @{-->}_{\tau}"D";"C"
\ar @{-->}_{\tau}"A";"E"
\ar @{-->}_{\tau}"E";"O"
\ar @{-->}_{\tau}"F";"D"
\end{xy}\quad \cdots
$$
In particular, $\mathcal{CH}_{\text{np}}\neq\varnothing$.
\item For any $n\in\mathbb{Z}_{\geq 1}$ and $\lambda\in\mathbb{P}^1(\kappa)$, there is an isomorphism $\tau Z^{\lambda}_n\simeq Z^{-\lambda}_n$, where we understand $-\infty=\infty$. In particular, $\mathcal{CH}_\text{np}$ consists of the unique component containing $Z_0$.
\item For any $m\in\mathbb{Z}$, the Heller lattice $Z_m$ appears on the boundary in $\mathcal{CH}_\text{np}$.
\end{enumerate}
\end{proposition}

First, we prove that the indecomposable Heller lattice $Z_m$ is not periodic in $\mathcal{CH}_{A}$. In order to do this, we introduce another $\mathcal{O}$-basis of $Z_m$ for each $m\leq 0$ as follows;
\begin{align*} 
Z_m&= \mathcal{O}\varepsilon e_1 \oplus \mathcal{O}(Xe_1-Ye_2) \oplus \mathcal{O} Ye_1\oplus \mathcal{O} XYe_1 \\
& \oplus \mathcal{O}\varepsilon e_2 \oplus \mathcal{O}(Xe_2-Ye_3) \oplus \mathcal{O} \varepsilon Ye_2 \oplus \mathcal{O} XYe_2 \\
& \oplus \cdots \\
& \oplus \mathcal{O}\varepsilon e_{|m|} \oplus \mathcal{O}(Xe_{|m|}-Ye_{|m|+1}) \oplus \mathcal{O} \varepsilon Ye_{|m|}\oplus \mathcal{O} XYe_{|m|} \\
& \oplus \mathcal{O}\varepsilon e_{|m|+1} \oplus \mathcal{O}Xe_{|m|+1} \oplus \mathcal{O} \varepsilon Ye_{|m|+1}\oplus \mathcal{O} XYe_{|m|+1}.
\end{align*}
We denote by $\overline{\mathbb{B}}(m)$ this $\mathcal{O}$-basis of $Z_m$.

\para \textbf{Proof of (1) in Proposition \ref{Heller2}.} We compute $\tau Z_m$ in the following five cases.
\[ (\mathrm{a})\ m=1,\quad (\mathrm{b})\ m>1,\quad  (\mathrm{c})\ m=0,\quad (\mathrm{d})\ m=-1,\quad (\mathrm{e})\ m<-1. \]

Suppose $(\mathrm{a})$. Since the projective cover of $Z_1$ is given by
\[ \pi_1 :   A\oplus A  \longrightarrow Z_1,\quad e_1  \longmapsto  \varepsilon e_1, \quad e_2  \longmapsto  XY e_1, \]
we have
\[ \tau Z_1=\mathcal{O}(-XYe_1+\varepsilon e_2)\oplus \mathcal{O}Xe_2\oplus \mathcal{O}Ye_2\oplus \mathcal{O}XYe_2\simeq Z_0.\] 

Suppose $(\mathrm{b})$. Since the projective cover of $Z_m$ is given by
\[ \begin{array}{ccccl}
\pi _m & : & A^{\oplus 2m-1} & \longrightarrow & \quad Z_m \\
 & & e_i & \longmapsto & \left\{\begin{array}{ll}
                                                 \varepsilon e_{k} & \text{if $i=2k-1$, $k=1,2,3,\ldots , m-1,$} \\
                                                 Ye_{k-1}-Xe_{k} & \text{if $i=2k-2$, $k=1,2,3,\ldots ,m,$} \\
                                                 -\varepsilon e_n & \text{if $i=2m-1$,}\end{array}\right. \end{array} \]
we have
\begin{align*} 
\tau Z_m=& \bigoplus _{k=1}^{m-2}\bigg(\mathcal{O}(Ye_{2k-1}-Xe_{2k+1}-\varepsilon e_{2k}) \oplus \mathcal{O}(XYe_{2k-1}-\varepsilon Xe_{2k}) \\
&\oplus \mathcal{O}(-Xe_{2k+2}-Ye_{2k}) \oplus \mathcal{O}(-XYe_{2k}) \bigg)  \\
&\oplus \mathcal{O}(Ye_{2m-3}+Xe_{2m-1}-\varepsilon e_{2m-2}) \oplus \mathcal{O}(XYe_{2m-3}-\varepsilon Xe_{2m-2}) \\
&\oplus \mathcal{O}(XYe_{2m-1}-\varepsilon Ye_{2m-2}) \oplus \mathcal{O}(-XYe_{2m-2}). \end{align*}
We change the above $\mathcal{O}$-basis of $\tau Z_m$ by using the invertible matrix $P=(P_{i,j})\in\mathrm{M}_{4m}(\mathcal{O})$ defined by $P_{i,j}:=(-1)^i\delta_{i,j}I_4$. Then, the representing matrices of the actions of $X$ and $Y$ on $\tau Z_m$ with respect to the new ordered $\mathcal{O}$-basis coincide with those on $Z_{m-1}$. It follows that $\tau Z_m\simeq Z_{m-1}$. 

Suppose $(\mathrm{c})$. Since the projective cover of $Z_0$ is given by
\[\pi_0  :   A\oplus A \oplus A  \longrightarrow  Z_0,\quad  e_1  \longmapsto  \varepsilon e_1, \ e_2  \longmapsto  Xe_1, \ e_3 \longmapsto  Ye_1,\]
we have an isomorphism 
\begin{align*}
\tau Z_0=&\ \mathcal{O}(-Ye_1+\varepsilon e_3) \oplus \mathcal{O} (-XYe_1+\varepsilon Xe_3) \oplus \mathcal{O}Ye_3 \oplus \mathcal{O}XYe_3\\
&\oplus\mathcal{O}(-Xe_1+\varepsilon e_2) \oplus \mathcal{O} Xe_2 \oplus \mathcal{O}(Ye_2- Xe_3) \oplus \mathcal{O}XYe_2 \\
\simeq &\ Z_{-1}. \end{align*} 

Next, we consider the case (d) and (e). The projective cover of $Z_{m}$ $(m\leq -1)$ is given by
\[ \begin{array}{cccl}
\pi _m  : & A^{\oplus 2|m|+3} & \longrightarrow & \quad Z_m \\
 &  e_i & \longmapsto & \left\{\begin{array}{ll}
                                                 \varepsilon e_{k} & \text{if $i=2k-1$, $k=1,2,\ldots , |m|+1,$} \\
                                                 Ye_1 & \text{if $i=2$,} \\
                                                 Ye_{k}-Xe_{k-1} & \text{if $i=2k$, $k=1,2,\ldots ,|m|-1,$} \\
                                                 Xe_{|m|+1} & \text{if $i=2|m|+2$,} \\
                                                 Ye_{|m|+1}-Xe_{|m|}& \text{if $i=2|m|+3$.}\end{array}\right. \end{array} \]
Thus, an $\mathcal{O}$-basis of $\tau Z_m$ is given as follows. If $m=-1$, then
\begin{align*}
 \tau Z_{-1}=&\ \mathcal{O}(\varepsilon e_2-Ye_1)\oplus\mathcal{O}(Xe_2+Ye_5)\oplus\mathcal{O}Ye_2\oplus\mathcal{O}XYe_2 \\
 & \oplus\mathcal{O}(Ye_3-Xe_1-\varepsilon e_5)\oplus\mathcal{O}(-Ye_4+Xe_5)\\
 &\oplus\mathcal{O}(XYe_1+\varepsilon Ye_5)\oplus\mathcal{O}XYe_5 \\
 & \oplus\mathcal{O}(Xe_3-\varepsilon e_4)\oplus\mathcal{O}Xe_4 \oplus\mathcal{O}(XYe_3+\varepsilon Ye_4)\oplus\mathcal{O}XYe_4,\end{align*}
and if $m<-1$, then 
\begin{align*}
\tau Z_m=&\ \mathcal{O}(Ye_1-\varepsilon e_2) \oplus\mathcal{O}(Ye_4+Xe_2) \oplus\mathcal{O}Ye_2\oplus\mathcal{O}XYe_2 \\
&\oplus\bigoplus _{k=1} ^{|m|-2}\bigg( \mathcal{O}(Ye_{2k+1}-Xe_{2k-1}-\varepsilon e_{2k+2})  \oplus \mathcal{O}(Ye_{2k+4}+Xe_{2k+2})\\
&\oplus\mathcal{O}(XYe_{2k-1}+\varepsilon Ye_{2k+2}) \oplus\mathcal{O}XYe_{2k+2}\bigg) \\
&\oplus\mathcal{O}(Ye_{2|m|+1}-Xe_{2|m|-1}-\varepsilon e_{2|m|+3}) \oplus\mathcal{O}(-Ye_{2|m|+2}-Xe_{2|m|+3})\\
& \oplus \mathcal{O}(XYe_{2|m|-1}+\varepsilon Ye_{2|m|+3}) \oplus\mathcal{O}XYe_{2|m|+3} \\
&\oplus\mathcal{O}(Xe_{2|m|+1}-\varepsilon e_{2|m|+2}) \oplus\mathcal{O}Xe_{2|m|+2}\\
& \oplus \mathcal{O}(XYe_{2|m|+1}-\varepsilon Ye_{2|m|+2}) \oplus\mathcal{O}XYe_{2|m|+2}. \end{align*}
We now consider the case $(\mathrm{d})$. Let $\widetilde{P}$ be the $12\times 12$ matrix defined by
\[ \widetilde{P}=\left( \begin{array}{ccc}
I_4 & 0 & 0 \\
0 & P & 0 \\
0 & 0 & P \end{array}\right), \]
where $P=\mathrm{diag}(-1,1,-1,1)$. If we change the above $\mathcal{O}$-basis of $\tau Z_{-1}$ by using $\widetilde{P}$, then the representing matrices of the actions of $X$ and $Y$ on $\tau Z_{-1}$ coincide with those on $Z_{-2}$ with respect to the $\mathcal{O}$-basis $\overline{\mathbb{B}}({-2})$. Thus, we have $\tau Z_{-1}\simeq Z_{-2}$. 

In the case $(\mathrm{e})$, we introduce a new ordered $\mathcal{O}$-basis of $\tau Z_m$ by using the invertible matrix $P=(P_{i,j})\in\mathrm{M}_{4(m+1)}(\mathcal{O})$ defined by
\[ P_{i,j}:=\left\{\begin{array}{ll}
(-1)^{i+1}\delta_{i,j}\mathrm{diag}(-1,1,1,1) & \text{if $(i,j)\neq (m+1,m+1)$}, \\
(-1)^m\mathrm{diag}(-1,1,-1,1) & \text{if $(i,j)=(m+1,m+1).$} \end{array}\right. \]
Then, the representing matrices of the actions of $X$ and $Y$ on $\tau Z_{m}$ with respect to the new ordered $\mathcal{O}$-basis coincide with those on $Z_{m-1}$ with respect to the $\mathcal{O}$-basis $\overline{\mathbb{B}}({m-1})$.

\para \textbf{Proof of (2) in Proposition \ref{Heller2}.}

We show that all indecomposable direct summands of the Heller lattice $Z_n^{\lambda}$ belong to a periodic component of $\mathcal{CH}_{A}$. To simplify the notation, we use the following symbols. If $n>1$, then 
\[ \left(\begin{array}{cccc}
\mathsf{b_{1,1}} & \mathsf{b_{1,2}} & \mathsf{b_{1,3}} & \mathsf{b_{1,4}} \\
\mathsf{b_{2,1}} & \mathsf{b_{2,2}} & \mathsf{b_{2,3}} & \mathsf{b_{2,4}} \\
\vdots & \vdots & \vdots & \vdots \\
\mathsf{b_{n-1,1}} & \mathsf{b_{n-1,2}} & \mathsf{b_{n-1,3}} & \mathsf{b_{n-1,4}} \\
\mathsf{b_{n,1}} & \mathsf{b_{n,2}} & \mathsf{b_{n,3}} & \mathsf{b_{n,4}}
\end{array}\right):=
 \left(\begin{array}{cccc}
\varepsilon e_1 & Xe_1 & (Ye_1-Xe_2) & XY e_1 \\
\varepsilon e_2 &\varepsilon Xe_2 & (Ye_2-Xe_3) & XY e_2 \\
\vdots & \vdots & \vdots & \vdots \\
\varepsilon e_{n-1} &\varepsilon Xe_{n-1} & (Ye_{n-1}-Xe_n) & XY e_{n-1} \\
\varepsilon e_n &\varepsilon Xe_n & \varepsilon Ye_n & XY e_n 
\end{array}\right). \]
If $n=1$, then
\[ (\mathsf{b_{1,1}},\mathsf{b_{1,2}},\mathsf{b_{1,3}},\mathsf{b_{1,4}}):=(\varepsilon e_1, Xe_1, \varepsilon Ye_1, XYe_1). \]
The actions $X$ and $Y$ on $Z_{n}^{\infty}$ are given by the following. If $n>1$, then
\[ X\mathsf{b_{i,j}}=\left\{\begin{array}{ll}
\varepsilon \mathsf{b_{1,2}} & \text{if $\mathsf{i}=\mathsf{j}=\mathsf{1},$} \\
\mathsf{b_{i,2}} & \text{if $\mathsf{i}\neq \mathsf{1},\ \mathsf{j}=\mathsf{1},$} \\
\mathsf{b_{i,4}} & \text{if $\mathsf{i}\neq \mathsf{n},\ \mathsf{j}=\mathsf{3},$} \\
\varepsilon \mathsf{b_{n,4}} & \text{if $\mathsf{i}=\mathsf{n},\ \mathsf{j}=\mathsf{3},$} \\
0 & \text{otherwise,} \end{array}\right.
\quad
 Y\mathsf{b_{i,j}}=\left\{\begin{array}{ll}
\varepsilon \mathsf{b_{i,3}}+\mathsf{b_{i+1,2}} & \text{if $\mathsf{i}\neq \mathsf{n},\ \mathsf{j}=\mathsf{1},$} \\
\mathsf{b_{n,3}} & \text{if $\mathsf{i}= \mathsf{n},\ \mathsf{j}=\mathsf{1},$} \\
\mathsf{b_{1,4}} & \text{if $\mathsf{i}=\mathsf{1},\ \mathsf{j}=\mathsf{2},$} \\
\varepsilon \mathsf{b_{i,4}} & \text{if $\mathsf{i}\neq \mathsf{1},\ \mathsf{j}=\mathsf{2},$} \\
- \mathsf{b_{i+1,4}} & \text{if $\mathsf{i}\neq \mathsf{n},\ \mathsf{j}=\mathsf{3},$} \\
0 & \text{otherwise.} 
\end{array}\right.\]
If $n=1$, then
\[ X\mathsf{b_{1,j}}=\left\{\begin{array}{ll}
\varepsilon \mathsf{b_{1,j+1}} & \text{if $\mathsf{j=1,3}$,}\\
0 & \text{otherwise,}
\end{array}\right.\quad
Y\mathsf{b_{1,j}}=\left\{\begin{array}{ll}
\mathsf{b_{1,j+2}} & \text{if $\mathsf{j=1,2}$,} \\
0 & \text{otherwise.} \end{array}\right. \]

The statement can be shown by using similar arguments to those in the proof of (1) in Proposition \ref{Heller2}. First, we prove the $\lambda=\infty$ case. The projective cover of $Z_n^{\infty}$ in $\mathsf{latt}$-$A$ is given by
\[ \begin{array}{ccccl}
\pi _{n,\infty} & : & A^{\oplus 2n} & \longrightarrow & \quad Z_n^{\infty} \\
 & & e_i & \longmapsto & \left\{\begin{array}{ll}
                                                 \mathsf{b_{1,1}}& \text{if $i=1,$} \\
                                                 \mathsf{b_{1,2}} & \text{if $i=2$,} \\
                                                 \mathsf{b_{k,3}} & \text{if $i=2\mathsf{k}+1$, $\mathsf{k=1,2,\ldots ,n-1},$} \\
                                                 \mathsf{b_{k,1}}& \text{if $i=2\mathsf{k}$, $\mathsf{k=2,3,\ldots ,n}.$} \\
                                                 \end{array}\right. \end{array} \]
Then, we have isomorphisms 
\[ \begin{array}{rcl}
\mathrm{Ker}(\pi_{1,\infty}) & \hspace{-3mm} = & \hspace{-3mm} \mathcal{O}(-Xe_{1}+\varepsilon e_{2}) \oplus \mathcal{O}Xe_{2} \oplus\mathcal{O}(-XYe_{1}+\varepsilon Ye_{2})\oplus\mathcal{O}XYe_{2} \\
 & \hspace{-3mm} \simeq &  \hspace{-3mm} Z_{1}^{\infty}, \\
\mathrm{Ker}(\pi_{2,\infty}) & \hspace{-3mm}  = & \hspace{-3mm}  \mathcal{O}(-Xe_{1}+\varepsilon e_{2}) \oplus \mathcal{O}Xe_{2}\oplus\mathcal{O}(-Xe_{3}+ Ye_{2})\oplus\mathcal{O}XYe_{2} \\
&& \hspace{-4mm} \oplus \mathcal{O}(-Ye_1+Xe_{4}+\varepsilon e_{3}) \oplus \mathcal{O}(-XYe_{1}+\varepsilon Xe_3) \oplus\mathcal{O}(XYe_{4}+ \varepsilon Ye_{3}) \hspace{-0.5mm}\oplus \hspace{-0.5mm}\mathcal{O}XYe_{3} \\
& \hspace{-3mm} \simeq &  \hspace{-3mm} Z_{2}^{\infty}.\end{array} \]  

Suppose that $n\geq 3$. Then, an $\mathcal{O}$-basis of the kernel of $\pi_{n,{\infty}}$ is given by
\begin{align*}
\mathrm{Ker}(\pi_{n,\infty})=&\ \mathcal{O}(Xe_1-\varepsilon e_2) \oplus\mathcal{O}Xe_2 \oplus\mathcal{O}(Ye_2-Xe_3)\oplus\mathcal{O}XYe_2 \\
&\oplus\mathcal{O}(Ye_1-Xe_4-\varepsilon e_3) \oplus\mathcal{O}(XYe_1-\varepsilon Xe_3) \oplus\mathcal{O}(Ye_3+Xe_5)\oplus\mathcal{O}XYe_3 \\
&\oplus\bigoplus _{k=2} ^{n-2}\bigg( \mathcal{O}(Ye_{2k}-Xe_{2(k+1)}-\varepsilon e_{2k+1}) \oplus \mathcal{O}(XYe_{2k}-\varepsilon Xe_{2k+1})\\
&\oplus\mathcal{O}(Ye_{2k+1}+Xe_{2k+3}) \oplus\mathcal{O}XYe_{2k+1}\bigg) \\
&\oplus\mathcal{O}(Ye_{2(n-1)}-Xe_{2n}-\varepsilon e_{2n-1}) \oplus\mathcal{O}(XYe_{2(n-1)}-\varepsilon Xe_{2n-1})\\
&\oplus \mathcal{O}(XYe_{2n}+\varepsilon Ye_{2n-1}) \oplus\mathcal{O}XYe_{2n-1}. \\
 \end{align*}    
Let $P=(P_{i,j})\in\mathrm{Mat}(\mathcal{O},4(n+1),4(n+1))$ be the invertible matrix defined by
\[ P_{i,j}:=\left\{\begin{array}{ll}
\mathrm{diag}(-1,1,1,1) & \text{if $(i,j)=(1,1)$}, \\
\delta_{i,j}(-1)^i\mathrm{diag}(-1,-1,1,1) & \text{otherwise}.\end{array}\right. \]
By changing the $\mathcal{O}$-basis of $\tau Z_n^{\infty}$ by $P$, we have an isomorphism $Z_{n}^{\infty}\simeq \mathrm{Ker}(\pi_{n,\infty})$.

Next, we prove the $\lambda\neq\infty$ case. The projective cover of $Z_n^{\lambda}$ is given by
\[ \begin{array}{ccccl}
\pi _{n,\lambda} & : & A^{\oplus 2n} & \longrightarrow & \quad Z_n^{\lambda} \\
 & & e_i & \longmapsto & \left\{\begin{array}{ll}
                                                 \varepsilon e_{k} & \text{if $i=2k-1$, $k=1,2,\ldots n$,} \\
                                                 Ye_1-\lambda Xe_1 & \text{if $i=2$,} \\
                                                 Ye_{k}-\lambda Xe_{k}-Xe_{k-1} & \text{if $i=2k$, $k=2,3,\ldots ,n$,} \\
                                              \end{array}\right. \end{array} \]
and hence an $\mathcal{O}$-basis of the kernel of $\pi _{n,{\lambda}}$ is given by 
\begin{align*}
 &\mathcal{O}(\varepsilon e_2-Ye_1+\lambda Xe_1) \oplus \mathcal{O} (\varepsilon Xe_2-XYe_1) \oplus \mathcal{O} (Ye_2+\lambda Xe_2) \oplus \mathcal{O}XY e_2 \\
&\oplus \mathcal{O}(\varepsilon e_4-Ye_3+\lambda Xe_3+Xe_1) \oplus \mathcal{O} (\varepsilon Xe_4-XYe_3) \oplus \mathcal{O} (Ye_4+\lambda Xe_4+Xe_2) \oplus \mathcal{O}XY e_4 \\
&\oplus \cdots \\
&\oplus \mathcal{O}(\varepsilon e_{2n}-Ye_{2n-1}+\lambda Xe_{2n-1}+Xe_{2n-3}) \oplus \mathcal{O} (\varepsilon Xe_{2n}-XYe_{2n-1})\\
& \oplus \mathcal{O} (Ye_{2n}+\lambda Xe_{2n}+Xe_{2n-2}) \oplus \mathcal{O}XY e_{2n}.\end{align*}  
Let $P$ be the invertible matrix $P=(P_{i,j})\in\mathrm{Mat}(\mathcal{O},4(n+1),4(n+1))$ defined by $P_{i,j}:=\delta_{i,j}(-1)^{i+1}I_4$. By similar arguments as before, we obtain an isomorphism $Z_n^{-\lambda}\simeq \mathrm{Ker}(\pi_{n,\lambda})$. This finishes the proof.

\para
The statements (1) and (2) in Proposition \ref{Heller2}  imply that $\mathcal{CH}_\text{np}$ consists of the unique component containing $Z_0$, and it contains the Heller lattice $Z_n$ for all $n\in\mathbb{Z}$. In this paper, we call $\mathcal{CH}_\text{np}$ the \textit{non-periodic Heller component} of $A$.

Finally, we prove that non-periodic Heller lattices appear on the boundary in $\mathcal{CH}_\text{np}$. In other words, for each Heller lattice $Z_n$, the middle term of $\mathscr{E}(Z_n)$ is indecomposable as an object of the projectively stable category $\underline{\mathsf{latt}}$-$A:=\mathsf{latt}\text{-}A/\mathsf{proj}\text{-}A$, where $\mathsf{proj}$-$A$ is the full subcategory of $\mathsf{latt}$-$A$ consisting of finitely generated projective $A$-modules. Since the Auslander--Reiten translation $\tau$ induces an automorphism of $\mathcal{CH}_\text{np}$, it is sufficient to consider the case of $n=1$.  Let $\mathcal{B}_2:=\{e_l,Xe_l,Ye_l, XYe_l\}_{l=1,2}$ be the $\mathcal{O}$-basis of $A\oplus A$. We fix the $\mathcal{O}$-bases $\mathcal{B}_2$  and $\mathbb{B}(1)$.  Recall that the projective cover of $Z_1$ is given by
\[ \pi _1:A\oplus A \longrightarrow Z_1,\quad e_1\longmapsto \varepsilon e,\quad e_2\longmapsto XYe. \]
Then, the representing matrix of $\pi_1$ is
\[ \left(\begin{array}{cccccccc}
1&0&0&0&0&0&0&0 \\
0&1&0&0&0&0&0&0 \\
0&0&1&0&0&0&0&0 \\
0&0&0&\varepsilon&1&0&0&0
\end{array}\right). \]
Let $\psi \in\Hom_{A}(Z_1,A\oplus A)$, and we write
\begin{align*}
\psi(\varepsilon e)=& \sum_{i=1}^{2}(a_{i1}e_i+a_{i2}Xe_i+a_{i3}Ye_i+a_{i4}XYe_i ),\\
\psi(XYe)= & \sum_{i=1}^{2}(b_{i1}e_i+b_{i2}Xe_i+b_{i3}Ye_i+b_{i4}XYe_i). \end{align*}
 Since $\varepsilon\psi(XYe)=XY\psi(\varepsilon e)$, the representing matrix of $\psi$  is
\[ \left(\begin{array}{cccc}
\varepsilon b_{14}&0&0&0 \\
a_{12}&\varepsilon b_{14}&0&0 \\
a_{13}&0&\varepsilon b_{14}&0 \\
a_{14}&a_{13}&a_{12}&b_{14} \\
\varepsilon b_{24}&0&0&0 \\
a_{22}&\varepsilon b_{24}&0&0 \\
a_{23}&0&\varepsilon b_{24}&0 \\
a_{24}&a_{23}&a_{22}&b_{24} 
\end{array}\right). \]
Thus, the set of endomorphisms of $Z_1$ factorizing through $\pi_1$ is
\[ \left\{\left. \left(\begin{array}{cccc}
\varepsilon \alpha&0&0&0 \\
\beta&\varepsilon \alpha&0&0 \\
\gamma&0&\varepsilon \alpha&0 \\
\varepsilon \delta&\varepsilon \gamma&\varepsilon \beta&\varepsilon \alpha \end{array}\right)\ \right|\ \ \alpha , \beta, \gamma, \delta \in\mathcal{O} \right\}. \]
On the other hand, the radical of the endomorphism ring of $Z_1$ is given by 
\[ \mathrm{radEnd}_A(Z_1)=\left\{ \left.\left(\begin{array}{cccc}
\varepsilon a&0&0&0 \\
b&\varepsilon a&0&0 \\
c&0&\varepsilon a&0 \\
d&\varepsilon c&\varepsilon b&\varepsilon a \end{array}\right)\ \right|\ a,b,c,d\in\mathcal{O}\right\}. \]
Therefore, we may take an endomorphism $\varphi$ which satisfies conditions $(\mathrm{i})$ and $(\mathrm{iii})$ in Proposition \ref{AKM} as
\[ \varphi =\left(\begin{array}{cccc}
0&0&0&0 \\
0&0&0&0 \\
0&0&0&0 \\
1&0&0&0 \end{array}\right), \]
and we consider the pullback diagram along $\pi_1$ and $\varphi$:
$$
\begin{xy}
(-20,0) *[o]+{0}="O",(0,0) *[o]+{Z_0}="A",(20,0)*[o]+{A\oplus A}="B",
(40,0)*[o]+{Z_1}="C",(40,15) *[o]+{Z_1}="D",(60,0) *[o]+{0}="O1",
(20,15)*[o]+{\overline{E}_1}="E1",(0,15) *[o]+{Z_0}="Z0",(60,15) *[o]+{0}="O2",
(-20,15) *[o]+{0}="O3"
\ar "O";"A"
\ar "A";"B"
\ar "B";"C"_{\pi_1}
\ar "D";"C"^{\varphi}
\ar "C";"O1"
\ar "O3";"Z0"
\ar "Z0";"E1"
\ar "E1";"D"
\ar "D";"O2"
\ar "E1";"B"
\ar @{-}@<0.5mm>"Z0";"A"
\ar @{-}@<-0.5mm>"Z0";"A"
\end{xy}$$
By Proposition \ref{AKM}, the upper short exact sequence is isomorphic to $\mathscr{E}(Z_1)$. Then, an $\mathcal{O}$-basis of $\overline{E}_1$ is given by
\begin{align*}
\overline{E}_1=&\{(f_1,f_2,x)\in A\oplus A\oplus Z_1 |\ \pi_1(f_1,f_2)=\varphi(x)\} \\
      =&\ \mathcal{O}(e_2+\varepsilon e_3)\oplus\mathcal{O}Xe_2\oplus\mathcal{O}Ye_2 \oplus\mathcal{O}XYe_2\\
      &\oplus\mathcal{O}(XYe_1+\varepsilon ^2e_3)\oplus\mathcal{O}(\varepsilon Xe_3)\oplus\mathcal{O}(\varepsilon Ye_3)\oplus\mathcal{O}(XYe_3). 
       \end{align*}
       
 \para \textbf{Proof of (3) in Proposition \ref{Heller2}.} Since $\tau Z_{n}=Z_{n-1}$ and $\tau$ is an autofunctor on the stable module category $\underline{\latt}$-$A$, it is enough to show that the $A$-lattice $\overline{E}_1$ has exactly one non-projective  indecomposable direct summand. We have isomorphisms       
\begin{align*} \overline{E}_1 \simeq &\ \mathcal{O}(e_2+\varepsilon e_3)\oplus\mathcal{O}(Xe_2+\varepsilon Xe_3)\oplus\mathcal{O}(Ye_2+\varepsilon Ye_3) \oplus\mathcal{O}(XYe_2+\varepsilon XYe_3)\\
      &\oplus\mathcal{O}(\varepsilon ^2e_3)\oplus\mathcal{O}(\varepsilon Xe_3)\oplus\mathcal{O}(\varepsilon Ye_3)\oplus\mathcal{O}(XYe_3) \\
\simeq &\ A \oplus\mathcal{O}\varepsilon^2e_3 \oplus\mathcal{O}\varepsilon Xe_3\oplus\mathcal{O}\varepsilon Ye_3\oplus\mathcal{O}XYe_3.\end{align*}
Let $E_1=\mathcal{O}\varepsilon^2e_3 \oplus\mathcal{O}\varepsilon Xe_3\oplus\mathcal{O}\varepsilon Ye_3\oplus\mathcal{O}XYe_3.$ Then, $E_1$ is not isomorphic to $A$. Since $E_1\otimes \mathcal{K}\simeq A\otimes \mathcal{K}$, $E_1$ is an indecomposable $A$-lattice, and we complete the proof of our claim.

\subsection{The middle term of AR sequences ending at non-periodic Heller lattices}

In this subsection, we show the following proposition.
\begin{proposition}\label{Z_n} Let $\mathscr{E}(Z_m): 0\to Z_{m-1}\to \overline{E}_m \to Z_m\to 0$ be the almost split sequence ending at $Z_m$. Then, the following statements hold.
\begin{enumerate}[(1)]
\item For $m\in\mathbb{Z}$, $\overline{E}_m$ is an indecomposable object in $\underline{\latt}$-$A$.
\item For $m\leq 0$, we have an isomorphism $\overline{E}_m\otimes\kappa\simeq M(m-1)^{\oplus 4}$ in $\latt$-$A$. 
\item For $m\leq 0$, $\overline{E}_m$ is a non-projective indecomposable $A$-lattice.
\end{enumerate}
\end{proposition} 

Let $X$ be an $A$-lattice and $\pi:P\to X$ the projective cover.
Let $Q\otimes\kappa\to X\otimes\kappa$ be the projective cover. Then $\mathrm{rank}\ Q\leq \mathrm{rank}\ P$. On the other hand, it lifts to $Q\to X$ and it is an epimorphism by Nakayama's lemma. Thus, we have $\mathrm{rank}\ Q=\mathrm{rank}\ P$ and $P\otimes\kappa$ is the projective cover of $X\otimes\kappa$.  Therefore, we have $\tau(X)\otimes\kappa\simeq \Omega(X\otimes\kappa)$ as objects in the stable module category $\underline{\mathsf{mod}}$-$A\otimes\kappa$, where $\Omega$ is the syzygy functor.

\begin{lemma}\label{proj1} For all $n\in\mathbb{Z}$, there is an isomorphism
\[ \Omega(M(n))\simeq M(n-1)\quad \text{in}\quad \underline{\mathsf{mod}}\text{-}A\otimes\kappa. \]
\end{lemma}
\begin{proof}  
Since $A\otimes\kappa$ is symmetric, the functor $\Omega:\underline{\mathsf{mod}}$-$A\otimes\kappa\to\underline{\mathsf{mod}}$-$A\otimes\kappa$ is an autofunctor. Let $\Omega^{-1}$ be the quasi-inverse of $\Omega$. Note that Remark \ref{ARseq} implies that there are isomorphisms
\[ \Omega ^2(M(l))\simeq M(l-2)\quad \text{in}\quad \underline{\mathsf{mod}}\text{-}A\otimes\kappa\]
for any $l$  since $A\otimes\kappa$ is symmetric.

First, we show that $\Omega(M(n))\simeq M(n-1)$ in $\underline{\mathsf{mod}}\text{-}A\otimes\kappa$ for $n\leq 0$ by induction on $n$. It is clear for $n=0$. Assume that the statement holds for $n\leq k\leq 0$. The induction hypothesis $\Omega(M(n))\simeq\Omega(n-1)$ implies
\[ \Omega(M(n-1))\simeq \Omega^2(M(n)) \simeq M(n-2) \quad \text{in}\quad \underline{\mathsf{mod}}\text{-}A\otimes\kappa, \]
and the statement is true for $n-1$.

Now, we show that $\Omega^{-1}(M(n))\simeq M(n+1)$ in $\underline{\mathsf{mod}}\text{-}A\otimes\kappa$ for $n\geq 0$ by induction on $n$. It is easy to check that $\Omega(M(1))\simeq M(0)$. Thus, the statement is true for $n=0$. Assume that the statement holds for $1\leq k\leq n$. The induction hypothesis $\Omega^{-1}(M(n))\simeq M(n+1)$ implies
\[ \Omega^{-1}(M(n+1))\simeq \Omega^{-2}(M(n))\simeq M(n+2) \quad \text{in}\quad \underline{\mathsf{mod}}\text{-}A\otimes\kappa, \]
and the statement is true for $n+1$.
\end{proof}

\para \textbf{Proof of Proposition \ref{Z_n}.}
(1) It is a direct consequence from the statement (3) in Proposition \ref{Heller2}. 

We show the statements (2) and (3) by induction. 
Since an $\mathcal{O}$-basis of $E_1$ is given by
\[ E_1=\mathcal{O}\varepsilon ^2e_1\oplus \mathcal{O}\varepsilon Xe_1\oplus\mathcal{O}\varepsilon Ye_1\oplus \mathcal{O}XYe_1,\]
the projective cover of $E_1$ is 
\[ \pi ^{E_1} :A^{\oplus 4}\to E_1,\quad e_1\mapsto \varepsilon^2 e_1,\quad e_2\mapsto \varepsilon Xe_1,\quad e_3\mapsto \varepsilon Ye_1,\quad e_4\mapsto XYe_1, \]
and an $\mathcal{O}$-basis of $\tau E_1$ is given by
\begin{align*} 
 \tau E_1 =&\ \mathcal{O}(Xe_1-\varepsilon e_2)\oplus \mathcal{O} Xe_2\oplus \mathcal{O} (Ye_2-\varepsilon e_4)\oplus \mathcal{O} XYe_2\\
& \oplus \mathcal{O} (Ye_1-\varepsilon e_3)\oplus \mathcal{O} (Xe_3-\varepsilon e_4)\oplus \mathcal{O} Ye_3\oplus \mathcal{O} XYe_3 \\
&\oplus \mathcal{O} (XYe_1-\varepsilon ^2e_4)\oplus \mathcal{O} Xe_4\oplus \mathcal{O} Ye_4\oplus \mathcal{O} XYe_4.
\end{align*}
Applying the functor $-\otimes \kappa$ to $\tau E_1$, we have an isomorphism
\begin{align*} 
\tau E_1\otimes\kappa &=
\left(\begin{xy}
(0,0)*[o]+{Ye_1-\varepsilon e_3}="u1",
(8,0)*[o]+{}="dammy1",
(30,0)*[o]+{XYe_1-\varepsilon ^2 e_4}="v1",
(19,0)*[o]+{}="dammy2",
(0,-8)*[o]+{Xe_1-\varepsilon e_2}="up",
(8,-8)*[o]+{}="dammy3",
\ar @{-->}"dammy3";"dammy2"
\ar "dammy1";"dammy2"
\end{xy}\right)
\oplus
\left(\begin{xy}
(0,0)*[o]+{Ye_2-\varepsilon e_4}="u1",
(8,0)*[o]+{}="dammy1",
(24,0)*[o]+{XYe_2}="v1",
(19,0)*[o]+{}="dammy2",
(0,-8)*[o]+{Xe_2}="up",
(8,-8)*[o]+{}="dammy3",
\ar @{-->}"up";"dammy2"
\ar "dammy1";"dammy2"
\end{xy}\right) \\
&\quad \oplus \left(\begin{xy}
(0,0)*[o]+{Ye_3}="u1",
(8,0)*[o]+{}="dammy1",
(24,0)*[o]+{XYe_3}="v1",
(19,0)*[o]+{}="dammy2",
(0,-8)*[o]+{Xe_3-\varepsilon e_4}="up",
(8,-8)*[o]+{}="dammy3",
\ar @{-->}"up";"dammy2"
\ar "u1";"dammy2"
\end{xy}\right)\oplus 
\left(\begin{xy}
(0,0)*[o]+{Ye_4}="u1",
(20,0)*[o]+{XYe_4}="v1",
(0,-8)*[o]+{Xe_4}="up",
\ar @{-->}"up";"v1"
\ar "u1";"v1"
\end{xy}\right) \\
& \simeq M(-1)^{\oplus 4}.
\end{align*}
On the other hand, the dimension of  $\overline{E}_0\otimes\kappa$ as $\kappa$-vector space is $12$ since the sequence
\[ 0\longrightarrow  Z_{-1} \otimes\kappa \longrightarrow \overline{E}_0\otimes\kappa \longrightarrow Z_0\otimes\kappa \longrightarrow 0\]
is exact. It implies that $\overline{E}_0$ has no projective direct summands. Thus, $\overline{E}_0$ is an indecomposable $A$-lattice such that it is isomorphic to $\tau E_1$.

Now, we assume that the statements (2) and (3) are true for $m+1\leq 0$.
By the induction hypothesis (3), $\tau \overline{E}_{m+1}$ is defined. Then, the statements (1) and (3) in Proposition \ref{Heller2}  imply that  there is an isomorphism $\overline{E}_m\simeq\tau\overline{E}_{m+1}$ in $\underline{\latt}$-$A$. By applying $-\otimes\kappa$ to the both sides, we have the following isomorphisms 
\[ \overline{E}_m\otimes \kappa \simeq \tau(\overline{E}_{m+1})\otimes\kappa \simeq \Omega(\overline{E}_{m+1}\otimes\kappa)\simeq \Omega(M(m)^{\oplus 4})\simeq M(m-1)^{\oplus 4}\]
in $\underline{\mathsf{mod}}$-$A\otimes\kappa$ from Lemma \ref{proj1} and the induction hypothesis (2). By comparing the non-projective direct summands of $\overline{E}_m\otimes \kappa$ and $M(m-1)^{\oplus 4}$, we have an isomorphism 
\[ \overline{E}_m\otimes \kappa \simeq M(m-1)^{\oplus 4}\oplus P \quad \text{in}\quad \mathsf{mod}\text{-}A\otimes\kappa,\]
where $P$ is a projective $A\otimes\kappa$-module.

On the other hand, since $\overline{E}_m$ is the middle term of $\mathscr{E}(Z_m)$, we have $\mathrm{rank}(\overline{E}_m)=-8m+12$, which equals to $\mathrm{dim}_\kappa(M(m-1)^{\oplus 4})$. Therefore, $P=0$, and the statements (2) and (3) are true for $m$.

\begin{corollary}\label{E otimes k}
Fix an integer $m$. Let $\overline{E}_m$ be the middle term of $\mathscr{E}(Z_m)$. Then, there is an isomorphism
\[ \overline{E}_m\otimes \kappa \simeq M(m-1)^{\oplus 4} \]
as objects in $\underline{\mathsf{mod}}$-$A\otimes\kappa$. 
\end{corollary}

\begin{proof}
By Proposition \ref{Z_n}, we may assume that $m$ is positive.
We show that  $\overline{E}_m\otimes \kappa \simeq M(m-1)^{\oplus 4}$ in $\underline{\mathsf{mod}}$-$A\otimes\kappa$ by induction. If $m=1$, then we have isomorphisms in $\underline{\mathsf{mod}}$-$A\otimes\kappa$
\[ M(0)^{\oplus 4} \simeq \Omega^{-1}(M(-1)^{\oplus 4})\simeq \Omega^{-1}(\overline{E}_0\otimes\kappa)\simeq \overline{E}_1\otimes\kappa, \]
where $\Omega^{-1}$ is the quasi-inverse of the autofunctor $\Omega$. Suppose that $\overline{E}_m\otimes \kappa \simeq M(m-1)^{\oplus 4}$ in $\underline{\mathsf{mod}}$-$A\otimes\kappa$. Then, we have 
\[ M(m)^{\oplus 4} \simeq \Omega^{-1}(M(m-1)^{\oplus 4})\simeq \Omega^{-1}(\overline{E}_m\otimes\kappa)\simeq \overline{E}_{m+1}\otimes\kappa \quad\text{in}\quad\underline{\mod}\text{-}A\otimes\kappa,  \]
which gives the desired conclusion.
\end{proof}

\subsection{Excluding the possibility $B_{\infty}$, $C_{\infty}$ and  $D_{\infty}$}
From now on, we denote by $E_n$ the unique non-projective indecomposable direct summand of $\overline{E}_n$. Let $\overline{F_n}$ be the middle term of $\mathscr{E}(E_n)$. The aim of this subsection is to show the following proposition.

\begin{proposition}\label{indecF}\label{F1}
For any $n\in\mathbb{Z}$, the non-projective indecomposable direct summands of $\overline{F}_n$ are $Z_{n-1}$ and an indecomposable $A$-lattice $F_n$. Moreover, for all $m$, neither  $Z_m$ nor $E_m$ are isomorphic to $F_n$.
\end{proposition}

It is enough to show the assertion for the case $\overline{F}_1$. By the proof of (3) in Proposition \ref{Heller2}, we have
\[ E_1=\mathcal{O}\varepsilon^2e \oplus\mathcal{O}\varepsilon Xe\oplus\mathcal{O}\varepsilon Ye\oplus\mathcal{O}XYe. \]
Since $\mathrm{dim}_\kappa(E_1\otimes\kappa)=4=\mathrm{dim}_\kappa(M(0)^{\oplus 4})$, we have an isomorphism $E_1\otimes\kappa\simeq M(0)^{\oplus 4}$ by Corollary \ref{E otimes k}.

We construct $\mathscr{E}(E_1)$. Let $\mathcal{B}_4:=\{e_l,Xe_l,Ye_l, XYe_l\}_{l=1,\ldots, 4}$ be the $\mathcal{O}$-basis of $A^{\oplus 4}$. We fix these $\mathcal{O}$-bases. Since the projective cover of $E_1$ is given by
\[ \pi^{E_1}:A^{\oplus 4}\longrightarrow E_1,\quad e_1\longmapsto \varepsilon ^2e,\ e_2\longmapsto \varepsilon Xe,\ e_3\longmapsto \varepsilon Ye,\ e_4\longmapsto  XYe,\]
the representing matrix of $\pi^{E_1}$ is 
\[ \left(\begin{array}{cccccccccccccccc}
1&0&0&0&0&0&0&0    &0&0&0&0&0&0&0&0  \\
0&\varepsilon&0&0&1&0&0&0 &0&0&0&0&0&0&0&0  \\
0&0&\varepsilon&0&0&0&0&0 &1&0&0&0&0&0&0&0 \\
0&0&0&\varepsilon^2&0&0&\varepsilon&0 &0&\varepsilon&0&0&1&0&0&0 
\end{array}\right). \]
On the other hand, 
the radical of $\End_A(E_1)$ is given by
\[ \mathrm{rad}\End_A(E_1)=\left\{ \left.\left(\begin{array}{cccc}
\varepsilon a&0&0&0 \\
b&\varepsilon a&0&0 \\
c&0&\varepsilon a&0 \\
d&\ c& b&\varepsilon a \end{array}\right)\ \right|\ a,b,c,d\in\mathcal{O}\right\}.\]
By similar arguments to 2.10 in the subsection 2.3, we obtain:

\begin{lemma}\label{factor}
Any endomorphism of $E_1$ which factors through $p$ is reprsented by
\[ \left(\begin{array}{cccc}
\varepsilon ^2a&0&0&0 \\
\varepsilon ^2b&\varepsilon ^2a&0&0 \\
\varepsilon ^2c&0&\varepsilon ^2a&0 \\
\varepsilon ^2d& \varepsilon^2c& \varepsilon^2b&\varepsilon ^2a \end{array}\right)\]
for some $a,b,c,d\in\mathcal{O}$. 
\end{lemma} 
\begin{proof}
The proof is straightforward.
\end{proof}

Let $\varphi :E_1\to E_1$ be the endomorphism defined by $\varphi(\varepsilon ^2e)=\varepsilon XYe$. Note that $\varphi(\varepsilon Xe)=\varphi(\varepsilon Ye)=\varphi(XYe)=0$. We consider the  pullback diagram along $\pi^{E_1}$ and $\varphi$:
\begin{equation}\label{a.s.s.2}
\begin{xy}
(-20,-7) *[o]+{0}="O",(0,-7) *[o]+{E_0}="A",(20,-7)*[o]+{A^{\oplus 4}}="B",
(40,-7)*[o]+{E_1}="C",(40,7) *[o]+{E_1}="D",(60,-7) *[o]+{0}="O1",
(20,7)*[o]+{\overline{F}_1}="E1",(0,7) *[o]+{E_0}="Z0",(60,7) *[o]+{0}="O2",
(-20,7) *[o]+{0}="O3"
\ar "O";"A"
\ar "A";"B"
\ar "B";"C"_{\pi^{E_1}}
\ar "D";"C"^{\varphi}
\ar "C";"O1"
\ar "O3";"Z0"
\ar "Z0";"E1"
\ar "E1";"D"
\ar "D";"O2"
\ar "E1";"B"
\ar @{-}@<0.5mm>"Z0";"A"
\ar @{-}@<-0.5mm>"Z0";"A"
\end{xy}
\end{equation}

\begin{lemma} The following statements hold.
\begin{enumerate}[(1)]
\item $\varphi$ does not factor through $\pi^{E_1}$.
\item For each $f\in\mathrm{rad}\mathrm{End}_A(E_1)$, $\varphi\circ f$ factors through $\pi^{E_1}$.
\end{enumerate} 
\end{lemma}
\begin{proof} 
(1) If $\varphi$ factors through $\pi^{E_1}$, then it contradicts with Lemma \ref{factor}. 

(2) Let $f\in\mathrm{rad}\mathrm{End}_{A}(E_{1})$. Assume that $f(\varepsilon^2 e)=\varepsilon a(\varepsilon^2e)+b(\varepsilon Xe)+c(\varepsilon Ye)+d(XYe)$ for some $a,b,c,d\in\mathcal{O}$.
Since $\varepsilon^2f(XYe)=XYf(\varepsilon^2 e)=\varepsilon^3aXYe$, we have $f(XYe)=\varepsilon aXYe$, and hence $\varphi\circ f(\varepsilon^{2}e)=\varepsilon ^{2}a(XYe)$. Define $\psi:E_{1}\to A^{{\oplus 4}}$ by
$\psi(\varepsilon^{2}e)=aXYe_{1}$. Then, it is easy to check $\varphi\circ f=\pi^{E_1}\circ\psi$. 
\end{proof}

By Proposition \ref{AKM}, the upper short exact sequence in (\ref{a.s.s.2}) is the almost split sequence ending at $E_1$. 

\para \textbf{Proof of Proposition \ref{indecF}.}
The $A$-lattice $\overline{F}_1$ is a direct sum of $F_1$ and $F_1'$, where  
\begin{align*}
F_1 = &\mathcal{O}(Xe_1-\varepsilon e_2)\oplus\mathcal{O}Xe_2\oplus \mathcal{O}(XYe_1-\varepsilon Ye_2)\oplus\mathcal{O}XYe_2 \\
 &\oplus\mathcal{O}(Ye_1-\varepsilon e_3)\oplus\mathcal{O}(Xe_3-Ye_2)\oplus\mathcal{O}Ye_3\oplus\mathcal{O}XYe_3 \\
 &\oplus\mathcal{O}(Xe_3+\varepsilon ^2e)\oplus\mathcal{O}\varepsilon Xe\oplus\mathcal{O}\varepsilon Ye\oplus\mathcal{O}XYe, \\
F_1' =&\mathcal{O}(\varepsilon e_4+\varepsilon ^2e)\oplus\mathcal{O}(Xe_4+\varepsilon Xe)\oplus\mathcal{O}(Ye_4+\varepsilon Ye)\oplus\mathcal{O}(XYe_4+\varepsilon XYe).\end{align*}
Obviously, the $A$-lattice $F_1'$ is isomorphic to the Heller lattice $Z_0$. We show that the $A$-lattice $F_1$ is indecomposable. The actions of $X$ and $Y$ on $F_1$ with respect to the above basis are given by the following matrices:
\[ X=\left(\begin{array}{cccccccccccc}
0 & 0 & 0 & 0                 &0 & 0& 0  & 0                &  &  &  &   \\
-\varepsilon & 0 & 0 & 0 & 0& 0 & 0 & 0 &  & &  &  \\
0 & 0 & 0 & 0                 & 1& 0  & 0 & 0            &    &  \hsymb{0} &  &    \\
0 & 0 & -\varepsilon & 0 &  0&-1 & 0 &  0               &  &  &  &                    \\

 &  &  &                        & 0 & 0 & 0 & 0 &  &  &  &    \\
 &  &  &        & -\varepsilon & 0 & 0 & 0 &  &  &  &      \\
 &  \hsymb{0}&  &        & 0 & 0 & 0 & 0 &  &\hsymb{0}  &  &     \\
 &  &  &                       & 0 & 0 & 1 & 0 &  &  &  &     \\

 &  &  &                &  &  &  &  & 0 & 0 & 0 & 0   \\
 &  & &  &             &  &  &  & \varepsilon  & 0 & 0 & 0   \\
 & \hsymb{0} &  & &&\hsymb{0}    &  &  & 0 & 0 & 0 & 0     \\
 &  & &                &  &  &  &  & 0 & 0 & \varepsilon & 0   \\

\end{array}\right)\]

\[  Y=\left(\begin{array}{cccccccccccc}
0 & 0 & 0 & 0                 &  &  &  &                 &  &  &  &    \\
0 & 0 & 0 & 0 &  &  &  &  &  & &  &   \\
1 & 0 & 0 & 0                 && \hsymb{0}   &  &    & & \hsymb{0}   &  & \\
0 & 1 & 0 & 0                 &  &  &  &                 &  &  &  &              \\

 &  &  &                        & 0 & 0 & 0 & 0 &  0&  0& 0 & 0   \\
 &  &  &                        & 0 & 0 & 0 & 0 & 0 & 0 & 0 & 0   \\
 &  \hsymb{0}&  &        & -\varepsilon & 0 & 0 & 0 & 0 & 0 & 0 & 0  \\
 &  &  &                       & 0 & 1 & 0 & 0 & 1 & 0 & 0 & 0 \\

 &  & &                &  &  &  &  & 0 & 0 & 0 & 0   \\
 &  & &                 &  &  &  &  & 0 & 0 & 0 & 0   \\
 &\hsymb{0}  & &  &  & \hsymb{0} &  &  & \varepsilon & 0 & 0 & 0  \\
 &  &  &                &  &  &  &  & 0 & \varepsilon & 0 & 0   \\
\end{array}\right)\]
Let $M=(x_{i,j})\in\End_A(F_1)$ be an idempotent. By the equalities $MX=XM$ and $MY=YM$, the idempotent $M$ is of the form $M=(M_1\ \  M_2)$,
where $M_1$ and $M_2$ are
\[ M_1=\left(\begin{array}{cccccc}
x_{1,1} & 0 & 0 & 0 & -\varepsilon x_{3,7} & 0  \\
x_{2,1} & x_{1,1} & 0 & 0 & -\varepsilon x_{4,7} & -\varepsilon x_{3,7}   \\
x_{3,1} & x_{3,2} & x_{1,1} & 0 & x_{3,5} & x_{3,6}  \\
x_{4,1} & x_{4,2} & x_{2,1} & x_{1,1} & x_{4,5} & x_{4,6}  \\

-\varepsilon x_{3,2} & 0 & 0 & 0 & x_{1,1}-\varepsilon x_{3,6}& 0 \\
x_{6,1} & -\varepsilon x_{3,2} & 0 & 0 & x_{6,5} & x_{1,1}-\varepsilon x_{3,6}   \\
-\varepsilon x_{8,2} & 0 & \varepsilon ^2x_{3,2} & 0 & x_{7,5}& \varepsilon x_{3,2}  \\
x_{8,1} & x_{8,2} & x_{8,3} & -\varepsilon x_{3,2} & x_{8,5} & x_{8,6} \\

x_{9,1} & 0 & 0 & 0 & x_{9,5} & 0    \\
x_{10,1} & -x_{9,1} & 0 & 0 & -x_{12,7} & -x_{9,5}   \\
x_{11,1} & 0 & \varepsilon x_{9,1} & 0 & x_{11,5} & x_{9,1} \\
x_{12,1} & -x_{11,1} & \varepsilon x_{10,1} & -\varepsilon x_{9,1}& x_{12,5} & x_{12,6} \end{array}
\right) \]
\[M_2=\left(\begin{array}{cccccc}
0 & 0 & \varepsilon x_{3,11} & 0 & 0 & 0 \\
0 & 0 & x_{2,9} & -\varepsilon x_{3,11} & 0 & 0 \\
 x_{3,7} & 0 & x_{3,9} & -x_{8,12} & x_{3,11} & 0 \\
 x_{4,7} & -\varepsilon x_{3,7} & x_{4,9} & x_{4,10} & x_{4,11} & -x_{3,11} \\
 0 & 0 & -\varepsilon x_{8,12} & 0 & 0 & 0\\
 0 & 0 & x_{6,9} & \varepsilon x_{8,12} & 0 & 0 \\
x_{1,1}-\varepsilon x_{3,6} & 0 & \varepsilon x_{8,10} & 0 & \varepsilon x_{8,12} & 0 \\
   x_{8,7}& x_{1,1}-\varepsilon x_{3,6} & x_{8,9} & x_{8,10} & x_{8,11} & x_{8,12}  \\
  0 & 0 & x_{9,9}& 0 &0 & 0 \\
0 & 0 & x_{10,9} &  x_{9,9} & 0 & 0 \\
 -x_{9,5} & 0 & x_{11,9} & 0 & x_{9,9} &  0  \\
x_{12,7} & -\varepsilon x_{9,5} & x_{12,9} & x_{11,9} & x_{12,11}&x_{9,9}
 \end{array}
\right) \]
such that 
\[ \left\{\begin{array}{l}
x_{2,9}+\varepsilon x_{3,7}-\varepsilon x_{4,11}  =0, \\
\varepsilon x_{3,1}-\varepsilon x_{4,2}+x_{6,1}  =0,\\
\varepsilon x_{3,9}+\varepsilon x_{4,10}+x_{6,9}=0, \\
x_{6,9}+x_{9,9} =x_{1,1}-\varepsilon x_{3,6}+\varepsilon x_{8,11}, \\
x_{6,1}-x_{8,3}+x_{9,1} =0, \end{array}\right.
\left\{\begin{array}{l}
x_{6,5}+\varepsilon x_{8,7}+ x_{9,5}  =0, \\
x_{7,5}- x_{8,3}+\varepsilon x_{8,6}  =0,\\
x_{9,5}+x_{10,9} -x_{12,11}=0, \\
-x_{10,1}+x_{11,5}+x_{12,6} =0. \end{array}\right.\]
Note that, it follows that we have $x_{6,9}\in\varepsilon\mathcal{O}$ and $x_{9,9}=x_{1,1}-\varepsilon f$ for some $f\in \mathcal{O}$. Since $M$ is an idempotent, the following equality holds:
\begin{equation}\label{x}
 x_{1,1}(1-x_{1,1})=\varepsilon x_{3,11}x_{9,1}+\varepsilon^2x_{3,2}x_{3,7}. \end{equation}

Assume that $x_{1,1}\equiv 0\ \mathsf{mod}\ \varepsilon\mathcal{O}$. By the assumption, the element $x_{9,9}$ belongs to $\varepsilon\mathcal{O}$. By comparing the $(9,1)$-entries and $(3,2)$-entries of $M$ and $M^2$, respectively, we have 
\begin{equation}\label{1} x_{9,1}  =x_{1,1}x_{9,1}+x_{9,1}x_{9,9}-\varepsilon x_{3,2}x_{9,5} \in \varepsilon\mathcal{O},
\end{equation}
\begin{equation}\label{3} 
x_{3,2}=x_{1,1}x_{3,2}+x_{1,1}x_{3,2}-\varepsilon x_{3,2}x_{3,6}+x_{9,1}x_{8,12}.
\end{equation}
It follows from $(\ref{1})$ and $(\ref{3})$ that the equality
\begin{equation}\label{4}
x_{3,2}(1-2x_{1,1}+\varepsilon x_{3,6}+\varepsilon x_{9,5}(1-x_{1,1}-x_{9,9})^{-1}x_{8,12})=0
\end{equation}
holds. Thus, the elements $x_{3,2}$ and $x_{9,1}$ are zero, and hence $x_{1,1}=0$. Let $\overline{M}$ be $M\ \mathsf{mod}\ \varepsilon \mathcal{O}$. As $M^2=M$, it suffices to show that $\overline{M}$ is the zero matrix to conclude that $M$ itself is the zero matrix. Let $e_i$ $(1\leq i\leq 12)$ be standard row vectors. Then, the span of $e_1,e_5,e_9$ is stable by $\overline{M}$ and the representing matrix is nilpotent. Thus, $e_i\overline{M}=0$ holds for $i=1,5,9$. From the equalities  
\[ e_2\overline{M}=\overline{x_{2,1}}e_1+\overline{x_{2,9}}e_9,\quad e_6\overline{M}=\overline{x_{6,1}}e_1-\overline{x_{6,9}}e_9,\text{  and  } e_7\overline{M}=\overline{x_{7,5}}e_5, \]
we also obtain $e_i\overline{M}=0$ for $i=2,6,7$. Then a similar argument shows $e_i\overline{M}=0$ for $i=10,11$, and then for $i=3,12$, and finally for $i=4,8$.

Assume that  $x_{1,1}\equiv 1\ \mathsf{mod}\ \varepsilon\mathcal{O}$. Recall that $I_{12}$ is the identity matrix of size $12$. Then, $I_{12}-M$ is an idempotent whose $(1,1)$-entry is zero modulo $\varepsilon \mathcal{O}$, and $M=I_{12}$ follows.

On the other hand, the induced sequence
\[ 0\longrightarrow E_0 \otimes\kappa \longrightarrow \overline{F}_1\otimes\kappa \longrightarrow E_1\otimes\kappa \longrightarrow 0 \]
splits by \cite[Proposition 4.5]{K}. Thus, there is an isomorphism
\[ \overline{F}_1\otimes\kappa \simeq M(0)^{\oplus 4}\oplus M(-1)^{\oplus 4}. \]
By Proposition \ref{Heller}, we have $F_1\otimes\kappa\simeq M(0)^{\oplus 3}\oplus M(-1)^{\oplus 3}$ as $F_1'\simeq Z_0$. It follows from Proposition \ref{Heller} and Corollary \ref{E otimes k} that  $F_1$ is neither isomorphic to $Z_m$ nor  $E_m$ for all $m$.
\section{The shape of the non-periodic Heller component}

We continue using the symbols of the previous section. In this section, we describe the shape of the non-periodic Heller component $\mathcal{CH}_\text{np}$.  Since the component $\mathcal{CH}_\text{np}$ has no loops by Lemma \ref{valuation_of_loops}, one can apply Theorem \ref{Ried} to $\mathcal{CH}_\text{np}$. Thus, there exist a directed tree $T$ and an admissible group $G$ such that $\mathcal{CH}_\text{np}\simeq \mathbb{Z}T/G$. 
We determine the directed tree $T$ and the admissible group $G$ in the rest of the paper.
The final result is given in Section \ref{main}. The aim of this section is to give the candidates for $T$ and to show $G$ is trivial.

\subsection{Non zero subadditive function on the component}

In this subsection, we show that the admissible group $G$ is trivial.  By Theorem \ref{Z}, if $\mathcal{CH}_\text{np}$ admits non-zero subadditive function $f:(\mathcal{CH}_\text{np})_0\to \mathbb{Z}_{\geq 0}$, then there is an isomorphism $\mathcal{CH}_\text{np}\simeq \mathbb{Z}\Delta$ for some valued quiver $\Delta$ since $\mathcal{CH}_\text{np}$ is not smooth. On the other hand, Theorem \ref{Ried} implies that the tree class is uniquely determined by $\mathcal{CH}_\text{np}$ and $G$ is unique up to conjugation. Thus, we have  $T=\Delta$ and $G$ is trivial if $\mathcal{CH}_\text{np}$ admits a non-zero subadditive function.  

Now, we introduce two functions.
Let $X$ be an indecomposable $A$-lattice. Define two functions $D$ and $\mathcal{R}$ by 
\[ D(X):=\sharp\{\text{non-projective indecomposable direct summands of $X\otimes\kappa$}\}, \]
\[ \mathcal{R}(X):=\text{the rank of $X$ as an $\mathcal{O}$-module}. \]
Let $X=X_1\oplus\cdots\oplus X_l$, where  $X_1,\ldots ,X_l$ are indecomposable $A$-lattices. Then, we also define 
\[ D(X):=\sum_{i=1}^{l}D(X_i),\quad \text{and} \quad \mathcal{R}(X):=\sum_{i=1}^l \mathcal{R}(X_i).\]

We denote by $\mathsf{add}(\mathcal{CH}_\text{np})$ the set of finite direct sums of $A$-lattices which belong to $\mathcal{CH}_\text{np}$.
Our goal of this subsection is to show the following proposition.

\begin{proposition}\label{add}
Let $D$ and $\mathcal{R}$ be the functions given as above. Then, the following statements hold.
\begin{enumerate}[(1)]
\item $D(Z_m)=2$ and $D(E_m)=4$ for any $m\in\mathbb{Z}$.
\item $D$ is additive on $\mathcal{CH}_\text{np}$ with $D=D\circ\tau$. In particular, the admissible group $G$ is trivial. 
\item Let $M\in\mathsf{add}(\mathcal{CH}_\text{np})$ with $D(M)=2$.  Then, $M$ is indecomposable.
\item The values of $D$ on $\mathsf{add}(\mathcal{CH}_\text{np})$ are even. 
\item The function $\mathcal{R}$ is additive along exact sequences. 
\end{enumerate}
\end{proposition}

\begin{lemma}\label{stable of D}
Let $\mathcal{C}$ be a component of the stable Auslander--Reiten quiver for $\mathsf{latt}^{(\natural)}$-$\Lambda$, where $\Lambda$ is a symmetric $\mathcal{O}$-order, and $D:\mathcal{C}_0\to \mathbb{Z}_{\geq 0}$ the function defined as above. Then, the equality $D(X)=D(\tau X)$ holds.
\end{lemma} 
\begin{proof}
Let $X\in\mathcal{C}$ and $\pi :P\to X$ the projective cover of $X$. Then, $\pi\otimes\kappa:P\otimes\kappa\to X\otimes\kappa$ is the projective cover by Nakayama's lemma. 

Since the functor $-\otimes \kappa$ is exact on lattices, we obtain an isomorphism $\tau X\otimes \kappa\simeq \Omega(X\otimes \kappa)$, where $\Omega$ is the first syzygy over $\mod$-$\Lambda\otimes\kappa$. It implies that the number of non-projective indecomposable direct summands of $\tau X\otimes \kappa$ equals to the number of non-projective indecomposable direct summands of $\Omega(X\otimes \kappa)$. As $\Omega$ is an autofunctor of the stable module category, $D(X)=D(\tau X)$ follows.
\end{proof}

\begin{lemma}\label{additive of D}
Let $\Lambda$ be a symmetric $\mathcal{O}$-order and $D$ the function defined as above. If a short exact sequence $0\to \tau L\to E\to L\to 0$ in $\mathsf{latt}^{(\natural)}$-$\Lambda$ is the almost split sequence ending at $L$, then the equality
\[ D(L)+D(\tau L)=D(E) \]
holds whenever $L$ is not isomorphic to any direct summand of the Heller lattices. 
\end{lemma}
\begin{proof} Let $L$ be an indecomposable $\Lambda$-lattice such that $L\otimes\mathcal{K}$ is projective as an $\Lambda\otimes\mathcal{K}$-module. Suppose that $L$ is not isomorphic to Heller lattices. By \cite[Proposition 4.5]{K}, the induced exact sequence 
\[ 0\rightarrow \tau L\otimes\kappa \rightarrow E\otimes\kappa \rightarrow L\otimes\kappa\rightarrow 0 \]  
splits, which gives the desired conclusion. \end{proof}

\begin{lemma}\label{non-zero}
Let $\mathcal{C}$ be a component of the stable Auslander--Reiten quiver for $\mathsf{latt}^{(\natural)}$-$\Lambda$, where $\Lambda$ is a symmetric $\mathcal{O}$-order, $X$ a vertex of $\mathcal{C}$. Then, $D(X)$ is a positive integer.
\end{lemma}
\begin{proof}
Suppose that $D(X)=0$. Let $P\to X$ be the projective cover of $X$. Then, $P\otimes\kappa$ is the projective cover and $X\otimes\kappa$ is projective as an $A\otimes\kappa$-module, we have $\tau X\otimes\kappa = 0$, a contradiction.
\end{proof}

\para \textbf{Proof of Proposition \ref{add}.} (1)  This is a direct consequence of Propositions \ref{Heller} and \ref{Z_n}.

(2) By Lemma \ref{additive of D}, it is enough to consider the case $X\simeq Z_n$ for some $n$. Since $\mathscr{E}(Z_n)$ is of the form
\[ 0 \longrightarrow Z_{n-1} \longrightarrow E_{n} \longrightarrow Z_n \longrightarrow 0,\]
we have $2D(Z_n)=D(E_n)$ by (1). Therefore, $D$ is additive with $D=D\circ\tau$.

(3) It is sufficient to show the case that $M$ is not isomorphic to any Heller lattice. Suppose that $M=M_1\oplus M_2$  for some non-zero $A$-lattices $M_{1}$ and $M_{2}$. Since $M\in\mathsf{add}(\mathcal{CH}_\text{np})$, the direct summands $M_1$ and $M_2$ are not projective. Thus, we have $D(M_1)=D(M_2)=1$ by Lemma \ref{non-zero}, and the $A$-lattices $M_1$ and $M_2$ are vertices of $\mathcal{CH}_\text{np}$. Then, it follows from Propositions \ref{Heller}, \ref{Z_n} and the proof of Lemma \ref{additive of D} that $M_i\otimes\kappa$ is isomorphic to $M(n_i)$ for some $n_i\in\mathbb{Z}$, a contradiction with Lemma \ref{divide}.

(4) Let $U\in\mathsf{add}(\mathcal{CH}_\text{np})$. In (3), we have proved that there exists an isomorphism 
\[ U\otimes \kappa\simeq \bigoplus_{j\in J}M(j)^{\oplus n_j}, \]
where $J$ is a finite set of $\mathbb{Z}$. This implies that the dimension of $U\otimes\kappa$ is the sum of finite odd numbers. Therefore, $D(U)$ is even by Lemma \ref{divide}.

(5) Obviously, the function $\mathcal{R}$ is additive for short exact sequences.

\subsection{Valencies of vertices in the component}

In this subsection, we observe the number of arrows from each vertex in $\mathcal{CH}_\text{np}$. From Proposition \ref{Heller2}, the Heller lattice $Z_n$ appears on the boundary in $\mathcal{CH}_\text{np}$, and from Proposition \ref{indecF}, we have
\[ \sharp\{\text{arrows starting at $E_n$ in $\mathcal{CH}_\text{np}$}\}=\sharp\{\text{arrows ending at $E_n$ in $\mathcal{CH}_\text{np}$}\}=2 \]
for all $n\in \Z$. Thus, the component $\mathcal{CH}_\text{np}$ admits the following valued subquiver with trivial valuations:
$$
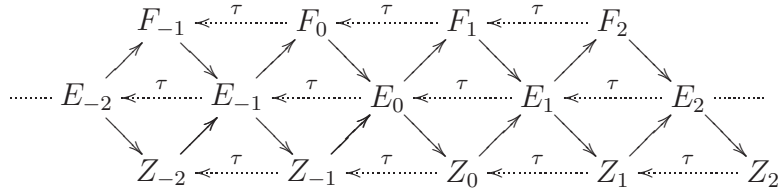
\begin{xy}
(0,0) *[o]+{Z_{-1}}="A",(20,0)*[o]+{Z_{0}}="B",
(40,0)*[o]+{Z_1}="C",(60,0) *[o]+{Z_2}="D",(70,0) *{}="E",(-20,0) *[o]+{Z_{-2}}="O",
(-10,10) *[o]+{E_{-1}}="e-1",(10,10)*[o]+{E_{0}}="e0",
(30,10)*[o]+{E_1}="e1",(50,10) *[o]+{E_2}="e2",(70,0) *{}="E",(-30,10) *[o]+{E_{-2}}="e-2",(60,10) *{}="OO",
(0,20) *[o]+{F_{0}}="f-1",(20,20)*[o]+{F_{1}}="f0",
(40,20)*[o]+{F_2}="f1",(70,20) *{}="ooo",(-20,20) *[o]+{F_{-1}}="f-2",(-40,10) *{}="OOO",
\ar @{..>}_{\tau}"B";"A"
\ar @{..>}_{\tau}"C";"B"
\ar @{..>}_{\tau}"D";"C"
\ar @{..>}_{\tau}"A";"O"
\ar @{..>}_{\tau}"e2";"e1"
\ar @{..>}_{\tau}"e1";"e0"
\ar @{..>}_{\tau}"e0";"e-1"
\ar @{..>}_{\tau}"e-1";"e-2"
\ar @{..}"OO";"e2"
\ar @{..}"OOO";"e-2"
\ar @{..>}_{\tau}"f1";"f0"
\ar @{..>}_{\tau}"f0";"f-1"
\ar @{..>}_{\tau}"f-1";"f-2"
\ar "O";"e-1"
\ar "e-1";"A"
\ar "A";"e0"
\ar "e0";"B"
\ar "B";"e1"
\ar "e1";"C"
\ar "C";"e2"
\ar "e2";"D"
\ar "e-2";"O"
\ar "O";"e-1"
\ar "e1";"f1"
\ar "A";"e0"
\ar "e-2";"f-2"
\ar "f-2";"e-1"
\ar "e-1";"f-1"
\ar "f-1";"e0"
\ar "e0";"f0"
\ar "f0";"e1"
\ar "f1";"e2"
\end{xy}$$

Given a vertex $X$of  $\mathcal{CH}_\text{np}$, we define a function $d$ on $\mathcal{CH}_\text{np}$ by
 \[ d(X):=\sharp\{\text{arrows from $X$ in $\mathcal{CH_\text{np}}$}\}. \]

In order to give candidates for the tree class $T$ of $\mathcal{CH}_\text{np}$, we introduce a pair of integers $(q(M), H(M))$ for $M\in \mathcal{CH}_\text{np}$ as follows. If $M$ is isomorphic to the Heller lattice $Z_n$, then $(q(M),H(M))=(1,n)$. Otherwise, we may choose $n$ such that a composition of irreducible morphisms $f_1\circ \cdots \circ f_k:Z_n\to M$ has the minimum length, and define $(q(M),H(M))=(k+1,n+k)$. For an $A$-lattice $M$, we also define the equilateral triangle $T(M)\subset \mathcal{CH}_\text{np}$ as follows:
\begin{itemize}
\item The vertices of $T(M)$ are $M$, $Z_n$ and $Z_{H(M)}$.
\item The edge $T(M)_1$ is a chain of irreducible morphisms from $Z_n$ to $M$.
\item The edge $T(M)_2$ is a chain of irreducible morphisms from $M$ to $Z_{H(M)}$.
\item The edge $T(M)_3$ is a chain of the Auslander--Reiten translation from $Z_{H(M)}$ to $Z_n$.
\end{itemize}
The set of vertices of $\mathcal{CH}_\text{np}$ is the disjoint union of the following three sets:
\begin{align*}
\mathcal{CH}_{\text{np}+}&=\{ X\in \mathcal{CH}_\text{np}\ |\ H(X)>0\},\\
\mathcal{CH}_{\text{np}0}&=\{ X\in \mathcal{CH}_\text{np}\ |\ H(X)=0\},\\
\mathcal{CH}_{\text{np}-}&=\{ X\in \mathcal{CH}_\text{np}\ |\ H(X)<0\}.
\end{align*}
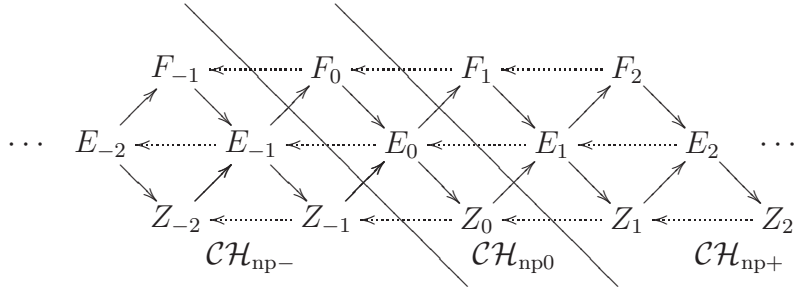
$$
\begin{xy}
(0,0) *[o]+{Z_{-1}}="A",(20,0)*[o]+{Z_{0}}="B",
(40,0)*[o]+{Z_1}="C",(60,0) *[o]+{Z_2}="D",(70,0) *{}="E",(-20,0) *[o]+{Z_{-2}}="O",
(-10,10) *[o]+{E_{-1}}="e-1",(10,10)*[o]+{E_{0}}="e0",
(30,10)*[o]+{E_1}="e1",(50,10) *[o]+{E_2}="e2",(70,0) *{}="E",(-30,10) *[o]+{E_{-2}}="e-2",(60,10) *{\cdots}="OO",
(0,20) *[o]+{F_{0}}="f-1",(20,20)*[o]+{F_{1}}="f0",
(40,20)*[o]+{F_2}="f1",(70,20) *{}="ooo",(-20,20) *[o]+{F_{-1}}="f-2",(-40,10) *{\cdots}="OOO",
(20,-10)*[o]+{}="dot1",(-20,30)*[o]+{}="dot2",(40,-10)*[o]+{}="dot3",(0,30)*[o]+{}="dot4",
(-10,-5)*[o]+{\mathcal{CH}_{\text{np}-}}="c-",(25,-5)*[o]+{\mathcal{CH}_{\text{np}0}}="c0",(55,-5)*[o]+{\mathcal{CH}_{\text{np}+}}="c+",
\ar @{..>}"B";"A"
\ar @{..>}"C";"B"
\ar @{..>}"D";"C"
\ar @{..>}"A";"O"
\ar @{..>}"e2";"e1"
\ar @{..>}"e1";"e0"
\ar @{..>}"e0";"e-1"
\ar @{..>}"e-1";"e-2"
\ar @{..>}"f1";"f0"
\ar @{..>}"f0";"f-1"
\ar @{..>}"f-1";"f-2"
\ar "O";"e-1"
\ar "e-1";"A"
\ar "A";"e0"
\ar "e0";"B"
\ar "B";"e1"
\ar "e1";"C"
\ar "C";"e2"
\ar "e2";"D"
\ar "e-2";"O"
\ar "O";"e-1"
\ar "e1";"f1"
\ar "A";"e0"
\ar "e-2";"f-2"
\ar "f-2";"e-1"
\ar "e-1";"f-1"
\ar "f-1";"e0"
\ar "e0";"f0"
\ar "f0";"e1"
\ar "f1";"e2"
\ar @{-}"dot1";"dot2"
\ar @{-}"dot3";"dot4"
\end{xy}$$

From now on, we assume that $\mathcal{CH}_\text{np}\neq\Z A_{\infty}$. Then, there exists an $A$-lattice $X$ such that 
\begin{enumerate}[(i)]
\item the $A$-lattice $X$ is not isomorphic to $Z_m$ and $E_m$ for all $m$.
\item the triangle $T(X)$ is contained in $\mathcal{CH}_{\text{np}-}$,
\item the number of outgoing arrows is two for each $A$-lattices on the edge $T(X)_1$ except for $Z_{H(X)-q(X)+1}$ and $X$, and  the number of indecomposable direct summands of $E_X$ is not $2$, where $E_X$ is the middle term of $\mathscr{E}(X)$.
\item valuations of arrows in the triangle $T(X)$ is trivial.
\end{enumerate}
$$
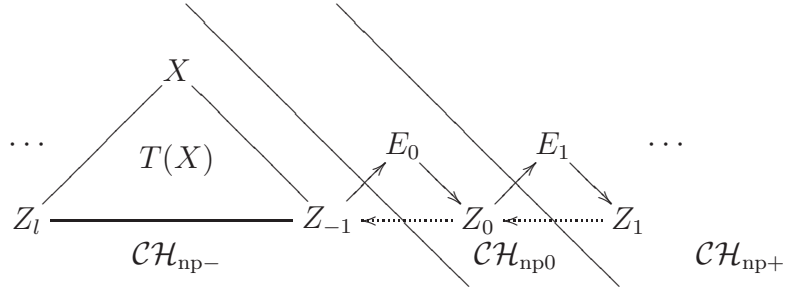
\begin{xy}
(0,0) *[o]+{Z_{-1}}="z-1",(20,0)*[o]+{Z_{0}}="B",(-40,0) *[o]+{Z_{l}}="zl",(-20,8)*[o]+{T(X)}="Tx",
(10,10)*[o]+{E_{0}}="e0",(45,10) *{\cdots}="OO",(-20,20)*[o]+{X}="x",(-40,10) *{\cdots}="OOO",
(40,0)*[o]+{Z_1}="C",(30,10)*[o]+{E_1}="e1",
(20,-10)*[o]+{}="dot1",(-20,30)*[o]+{}="dot2",(40,-10)*[o]+{}="dot3",(0,30)*[o]+{}="dot4",
(-20,-5)*[o]+{\mathcal{CH}_\text{np$-$}}="c-",(25,-5)*[o]+{\mathcal{CH}_\text{np$0$}}="c0",(55,-5)*[o]+{\mathcal{CH}_\text{np$+$}}="c+",

\ar "e0";"B"
\ar "e1";"C"
\ar "z-1";"e0"
\ar "B";"e1"
\ar @{..>}"B";"z-1"
\ar @{..>}"C";"B"

\ar @{-}"x";"zl"
\ar @{-}"zl";"z-1"
\ar @{-}"z-1";"x"

\ar @{-}"dot1";"dot2"
\ar @{-}"dot3";"dot4"
\end{xy}$$

It follows from Proposition \ref{add} that $D(M)=2q(M)$ for any $M\in T(X)$. Using Proposition \ref{add} and results from Section $2$, we may assume that $q(X)\geq 3$ and $H(X)=-1$. We set $q(X)=q$. 

Assume that $\mathscr{E}(X)$ is given by
\[ \mathscr{E}(X):\quad 0\longrightarrow \tau X \longrightarrow \bigoplus _{i=1}^{p} W_{i} \longrightarrow X\longrightarrow 0, \]
where $W_p\in T(X)$. Then, the neighborhood of $X$ in $\mathcal{CH}_\text{np}$ is given as follows.
\begin{equation}\label{nbdX}
\begin{xy}
(-5,-20) *[o]+{\tau W_p}="A",(25,-20)*[o]+{W_p}="B",
(-20,-5) *[o]+{\tau^2 X}="e-1",(10,-5)*[o]+{\tau X}="e0",
(40,-5)*[o]+{X}="e1",
(-5,7) *[o]+{\tau W_2}="tw21",(25,7)*[o]+{W_2}="w21",
(-5,0) *[o]+{\vdots}="dot1",(25,0)*[o]+{\vdots}="dot2",
(-5,-12) *[o]+{\tau W_{p-1}}="tw2",(25,-12)*[o]+{W_{p-1}}="w2",
(-5,15) *[o]+{\tau W_1}="tw1",(25,15)*[o]+{W_1}="w1",
\ar @{..>}^{\tau}"B";"A"
\ar @{..>}^{\tau}"e1";"e0"
\ar @{..>}^{\tau}"e0";"e-1"
\ar @{..>}^{\tau}"w2";"tw2"
\ar @{..>}_{\tau}"w21";"tw21"
\ar @{..>}_{\tau}"w1";"tw1"
\ar "e-1";"A"
\ar "A";"e0"
\ar "e0";"B"
\ar "B";"e1"
\ar "A";"e0"
\ar "e-1";"tw1"
\ar "tw1";"e0"
\ar "e0";"w1"
\ar "w1";"e1"
\ar "e-1";"tw2"
\ar "tw2";"e0"
\ar "e-1";"tw21"
\ar "tw21";"e0"
\ar "e0";"w2"
\ar "w2";"e1"
\ar "e0";"w21"
\ar "w21";"e1"
\end{xy}
\end{equation}
Here, we allow the possibility that $W_i\simeq W_k$ for some $i\neq k$ instead of writing the valuation. If $D(W_i)=s_i$, then the values of $D$ of $(\ref{nbdX})$ are as follows:
\begin{equation}\label{nbdX2}
\begin{xy}
(-5,-20) *[o]+{2(q-1)}="A",(25,-20)*[o]+{2(q-1)}="B",
(-20,-5) *[o]+{2q}="e-1",(10,-5)*[o]+{2q}="e0",
(40,-5)*[o]+{2q}="e1",
(-5,7) *[o]+{s_2}="tw21",(25,7)*[o]+{s_2}="w21",
(-5,0) *[o]+{\vdots}="dot1",(25,0)*[o]+{\vdots}="dot2",
(-5,-12) *[o]+{s_{p-1}}="tw2",(25,-12)*[o]+{s_{p-1}}="w2",
(-5,15) *[o]+{s_1}="tw1",(25,15)*[o]+{s_1}="w1",
\ar @{..>}^{\tau}"B";"A"
\ar @{..>}^{\tau}"e1";"e0"
\ar @{..>}^{\tau}"e0";"e-1"
\ar @{..>}^{\tau}"w2";"tw2"
\ar @{..>}_{\tau}"w21";"tw21"
\ar @{..>}_{\tau}"w1";"tw1"
\ar "e-1";"A"
\ar "A";"e0"
\ar "e0";"B"
\ar "B";"e1"
\ar "A";"e0"
\ar "e-1";"tw1"
\ar "tw1";"e0"
\ar "e0";"w1"
\ar "w1";"e1"
\ar "e-1";"tw2"
\ar "tw2";"e0"
\ar "e-1";"tw21"
\ar "tw21";"e0"
\ar "e0";"w2"
\ar "w2";"e1"
\ar "e0";"w21"
\ar "w21";"e1"
\end{xy}
\end{equation}

\begin{lemma}\label{D}  The following statements hold:
\begin{enumerate}[(1)]
\item The sum of $s_1,s_2,\ldots s_{p-2}$ and $s_{p-1}$ is $2(q+1)$.
\item The inequality $s_i\geq q$ is satisfied for any $i$.
\end{enumerate}
\end{lemma}
\begin{proof} (1) By Lemma \ref{additive of D}, we have
\[ 4q=\sum_{i=1}^{p-1}D(W_i)+D(W_p)=\sum_{i=1}^{p-1}s_i+2(q-1). \]
It follows that (1) holds.

(2) From Proposition \ref{add} (2) and $(\ref{nbdX2})$, we obtain that $2s_i\geq 2q$.
\end{proof}

\begin{lemma}\label{W} Suppose that $q<\infty$. Then, $d(X)$ is precisely three.  
\end{lemma}
\begin{proof}  Lemma \ref{D} implies that
\[ 2(q+1)=\sum_{i=1}^{p-1}s_i\geq (p-1)q. \]
Thus, the inequality $-2\leq q(3-p)$ holds. Since $p$ and $q$ are positive, we have $p=1,2,3$. If $p=1$, then $q=-1$ from Lemma \ref{D} (1), a contradiction. If $p=2$, then $s_1=2(q+1)$, which contradicts with the maximality of $q$ namely, the condition (iii). Therefore, we have $p=3$. Then, we may assume that $\mathscr{E}(X)$ is of the form
\[ 0\longrightarrow \tau X\longrightarrow  W_1\oplus  W_2\oplus  Y \longrightarrow X\longrightarrow 0 \]
with $Y\in T(X)$. We show that the three non-projective indecomposable $A$-lattices $W_1$, $W_2$ and $Y$ are pairwise non-isomorphic.

Suppose that $Y\simeq W_i$ for some $i$. Since $Y\in T(X)$, there exist arrows in $T(X)$ such that their valuations are not trivial, a contradiction. 

Suppose that $W_1\simeq W_2$. Then, the neighborhood of $X$ in $\mathcal{CH}_\text{np}$ is the following valued quiver:
$$
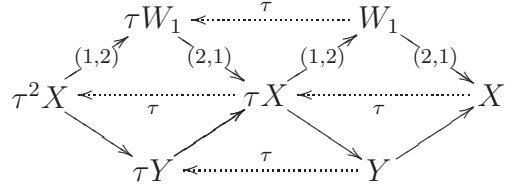
\begin{xy}
(0,0) *[o]+{\tau Y}="tY",(30,0)*[o]+{Y}="Y",
(-15,10) *[o]+{\tau^2 X}="e-1",(15,10)*[o]+{\tau X}="e0",
(45,10)*[o]+{X}="e1",
(0,20) *[o]+{\tau W_1}="tw1",(30,20)*[o]+{W_1}="w1",
\ar @{..>}_{\tau}"Y";"tY"
\ar @{..>}^{\tau}"e1";"e0"
\ar @{..>}_{\tau}"w1";"tw1"
\ar @{..>}^{\tau}"e0";"e-1"
\ar "e-1";"tY"
\ar "tY";"e0"
\ar "e0";"Y"
\ar "Y";"e1"
\ar "tY";"e0"
\ar "e-1";"tw1"|{(1,2)}
\ar "tw1";"e0"|{(2,1)}
\ar "e0";"w1"|{(1,2)}
\ar "w1";"e1"|{(2,1)}
\end{xy}$$
Indeed, if we write the value $W_1 \xrightarrow{(a,b)} X$, then clearly $a=2$ by the assumption. Thus, $\mathscr{E}(X)$ becomes
\[ 0\longrightarrow \tau X\longrightarrow W_1^{\oplus 2}\oplus Y \longrightarrow X\longrightarrow 0 \]
and we have $D(W_1)=q+1$ from Lemma \ref{additive of D}. Suppose that the almost split sequence ending at $W_1$ is 
\[ \mathscr{E}(W_1):\quad 0\longrightarrow  \tau W_1\longrightarrow \tau X^{\oplus b}\oplus  U_1 \longrightarrow W_1\longrightarrow 0, \]
where $U_1$ is an $A$-lattice. If $U_1=0$, then Lemma \ref{additive of D} implies that
\[ q+1= D(W_1) =qb, \]
hence $q(b-1)=1$, which contradicts with $q\geq 3$. Thus, $U_1\neq 0$ and $q(b-1)<1$. Since $b\geq 1$, we have $b=1$.

From the almost split sequence $\mathscr{E}(W_1)$, we have $D(U_1)=2$, and it implies that $U_1$ is indecomposable by Proposition \ref{add}. 
Therefore, we have $q=3$ from the inequality
\[ 4=D(U_1)+D(\tau U_1) \geq D(\tau W_1) =q+1. \]
Note that $\mathcal{CH}_\text{np}$ is the following valued stable translation quiver.
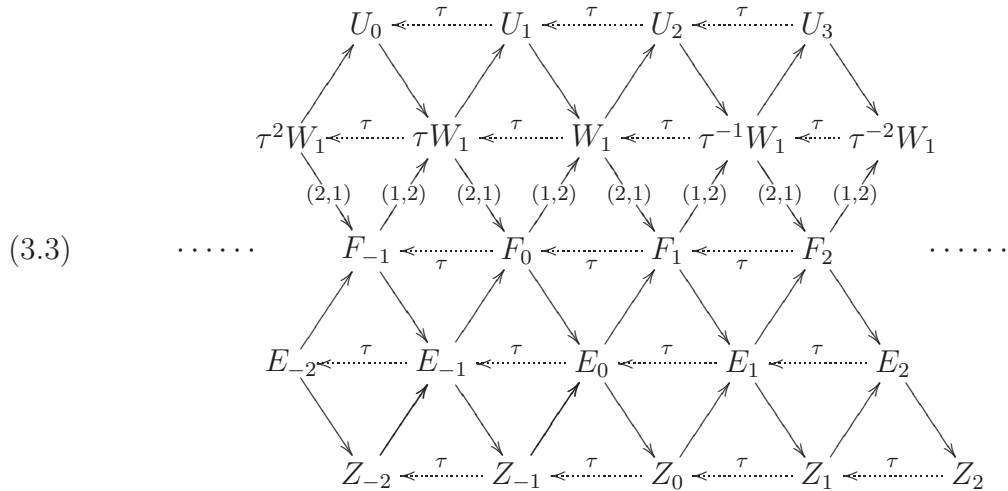
\begin{equation}\label{w2}
\begin{xy}
(0,-30) *[o]+{Z_{-1}}="A",(20,-30)*[o]+{Z_{0}}="B",
(40,-30)*[o]+{Z_1}="C",(60,-30) *[o]+{Z_2}="D",(70,-30) *{}="E",(-20,-30) *[o]+{Z_{-2}}="O",
(-10,-15) *[o]+{E_{-1}}="e-1",(10,-15)*[o]+{E_{0}}="e0",
(30,-15)*[o]+{E_1}="e1",(50,-15) *[o]+{E_2}="e2",(70,0) *{}="E",(-30,-15) *{E_{-2}}="e-2",(60,0) *{\cdots\cdots}="OO",
(0,0) *[o]+{F_{0}}="f-1",(20,0)*[o]+{F_{1}}="f0",
(40,0)*[o]+{F_2}="f1",(70,0) *{}="ooo",(-20,0) *[o]+{F_{-1}}="f-2",(-40,0) *{\cdots\cdots}="OOO",
(-10,15) *[o]+{\tau W_{1}}="w0",(10,15)*[o]+{W_{1}}="w1",
(30,15)*[o]+{\tau^{-1}W_1}="w2",(50,15) *[o]+{\tau ^{-2}W_1}="w3",(-30,15) *{\tau ^2W_{1}}="w-1",
(0,30) *[o]+{U_{1}}="u1",(20,30)*[o]+{U_2}="u2",
(40,30)*[o]+{U_3}="u3",(-20,30) *[o]+{U_0}="u0"
\ar @{..>}_{\tau}"B";"A"
\ar @{..>}_{\tau}"C";"B"
\ar @{..>}_{\tau}"D";"C"
\ar @{..>}_{\tau}"A";"O"
\ar @{..>}_{\tau}"e2";"e1"
\ar @{..>}_{\tau}"e1";"e0"
\ar @{..>}_{\tau}"e0";"e-1"
\ar @{..>}_{\tau}"e-1";"e-2"
\ar @{..>}_{\tau}"w3";"w2"
\ar @{..>}_{\tau}"w2";"w1"
\ar @{..>}_{\tau}"w1";"w0"
\ar @{..>}_{\tau}"w0";"w-1"
\ar @{..>}_{\tau}"u3";"u2"
\ar @{..>}_{\tau}"u2";"u1"
\ar @{..>}_{\tau}"u1";"u0"
\ar @{..>}^{\tau}"f1";"f0"
\ar @{..>}^{\tau}"f0";"f-1"
\ar @{..>}^{\tau}"f-1";"f-2"
\ar "O";"e-1"
\ar "e-1";"A"
\ar "A";"e0"
\ar "e0";"B"
\ar "B";"e1"
\ar "e1";"C"
\ar "C";"e2"
\ar "e2";"D"
\ar "e-2";"O"
\ar "O";"e-1"
\ar "e1";"f1"
\ar "A";"e0"
\ar "e-2";"f-2"
\ar "f-2";"e-1"
\ar "e-1";"f-1"
\ar "f-1";"e0"
\ar "e0";"f0"
\ar "f0";"e1"
\ar "f1";"e2"
\ar "w-1";"f-2"|{(2,1)}
\ar "f-2";"w0"|{(1,2)}
\ar "w0";"f-1"|{(2,1)}
\ar "f-1";"w1"|{(1,2)}
\ar "w1";"f0"|{(2,1)}
\ar "f0";"w2"|{(1,2)}
\ar "w2";"f1"|{(2,1)}
\ar "f1";"w3"|{(1,2)}
\ar "w-1";"u0"
\ar "u0";"w0"
\ar "w0";"u1"
\ar "u1";"w1"
\ar "w1";"u2"
\ar "u2";"w2"
\ar "w2";"u3"
\ar "u3";"w3"
\end{xy}
\end{equation}
It follows from Propositions \ref{Heller}, \ref{Z_n} and the proof of Lemma \ref{additive of D} that there is an isomorphism
\[ \tau W_1^{\oplus 2}\otimes \kappa \simeq M(-3)^{\oplus 3}\oplus M(-2)^{\oplus 2} \oplus M(-1)^{\oplus 3}, \]
a contradiction.  
\end{proof}

\section{Main results}\label{main}

In this section, we continue using the symbols and the assumption of the previous section. 
By Proposition \ref{Z_n}, $\mathcal{CH}_{\text{np}}$ is not smooth. From the results in Subsection 3.1, the function $D$ is subadditive with $D=D\circ \tau$ on $\mathcal{CH}_{\text{np}}$. Therefore, there exists a directed tree $T$ such that $\mathcal{CH}_{\text{np}}=\Z T$.
Since $D$ is additive with $D(X)=D(\tau X)$ for all $X\in\mathcal{CH}_\text{np}$, it satisfies
\[ 2D(X) = \sum_{Y\to X \text{ in }T}d_{YX}D(Y)+\sum_{X\to Y\text{ in }T}d'_{XY}D(Y)\quad X\in T. \] 
Thus, the tree class $\overline{T}$ of $\mathcal{CH}_\text{np}$ is one of infinite Dynkin diagrams or Euclidean diagrams from Theorem \ref{tree class}. 

\begin{lemma}\label{additive R}
Suppose that $X\in\mathcal{CH}_\text{np}$ is not isomorphic to $Z_n$ for all $n$. Then, the middle term of $\mathscr{E}(X)$ has no projective modules as direct summands.
\end{lemma}
\begin{proof}  By the proof of (3) in Proposition \ref{add}, $X\otimes \kappa$ and $\tau X\otimes\kappa$ have no projective modules as direct summands. Since $X$ is not a Heller lattice, the induced exact sequence $\mathscr{E}(X)\otimes\kappa: 0\to \tau X\otimes\kappa\to E_X\otimes\kappa\to X\otimes\kappa\to 0$ splits.
\end{proof}

Now, the following theorem can be proved.

\begin{maintheorem}  Let $\mathcal{O}$ be a complete discrete valuation ring, $A=\mathcal{O}[X,Y]/(X^2,Y^2)$ and $\Gamma_s(A)$ the stable Auslander--Reiten quiver for $\mathsf{latt}^{(\natural)}$-$A$. Assume that the residue field $\kappa$ is algebraically closed. Then, the following statements hold.
\begin{enumerate}[(1)]
\item  Let $M$ be an indecomposable $A\otimes_\mathcal{O}\kappa$-module. Then the Heller lattice of $M$ lies on a non-periodic component of $\Gamma_s(A)$ if and only if $M$ is given by a string path of even length.
\item  $\Gamma_s(A)$ contains a unique connected non-periodic Heller component $\mathcal{CH}_\text{np}$.
\item The component $\mathcal{CH}_\text{np}$ is isomorphic to $\mathbb{Z}A_{\infty}$. 
\item Every non-periodic indecomposable Heller lattice appears on the boundary of the component $\mathcal{CH}_\text{np}$.
\end{enumerate}
\end{maintheorem}

\begin{proof}
The statements (1), (2) and (4) had been proved in Proposition \ref{Heller2}. We only need to show the statement (3). Assume that $\overline{T}\neq A_{\infty}$. It implies from Propositions \ref{Heller2} and \ref{indecF} that $\overline{T}$ is one of $\widetilde{E_6}$, $\widetilde{E}_7$, $\widetilde{E}_8$, $\widetilde{F}_{41}$ or $\widetilde{F}_{42}$. On the other hand, Lemma \ref{W} implies that $\overline{T}$ is neither $\widetilde{F}_{41}$ nor $\widetilde{F}_{42}$.

First, we suppose that $\mathcal{CH}_\text{np}=\Z\widetilde{E}_6$. Then, $\mathcal{CH}_\text{np}$ has the following subquiver with bounds $U_n$ and $V_n$:
\begin{equation}\label{main1}
\begin{xy}
(-15,-20) *[o]+{Z_{-1}}="A",(15,-20)*[o]+{Z_{0}}="B",
(45,-20)*[o]+{Z_1}="C",(-45,-20)*[o]+{Z_{-2}}="z-2",
(-30,-10) *[o]+{E_{-1}}="e-1",(0,-10)*[o]+{E_{0}}="e0",
(30,-10)*[o]+{E_1}="e1",
(-15,0) *[o]+{F_{0}}="f-1",(15,0)*[o]+{F_{1}}="f0",
(-45,0) *[o]+{F_{-1}}="f-2",(45,0) *[o]+{F_{2}}="f2",
(-30,10) *[o]+{W_{0}}="w0",(0,10)*[o]+{W_{1}}="w1",(30,10)*[o]+{W_{2}}="w-2",
(-30,0) *[o]+{W_{0}'}="ww1",(0,0)*[o]+{W_{1}'}="ww0",(30,0)*[o]+{W_{2}'}="ww-2",
(-15,20) *[o]+{U_{1}}="u1",(15,20) *[o]+{U_{2}}="u2",
(-15,10) *[o]+{V_1}="v0",(15,10) *[o]+{V_{2}}="v1",
(-60,0)*[o]+{\cdots}="dot1",(60,0)*[o]+{\cdots}="dot2",

\ar "e-1";"A"
\ar "A";"e0"
\ar "e0";"B"
\ar "B";"e1"
\ar "e1";"C"
\ar "A";"e0"
\ar "f-2";"e-1"
\ar "e-1";"f-1"
\ar "f-1";"e0"
\ar "e0";"f0"
\ar "f0";"e1"
\ar "z-2";"e-1"
\ar "f-2";"w0"
\ar "w0";"f-1"
\ar "f-1";"w1"
\ar "w1";"f0"
\ar "w0";"u1"
\ar "u1";"w1"

\ar "f-2";"ww1"
\ar "ww1";"f-1"
\ar "f-1";"ww0"
\ar "ww0";"f0"
\ar "ww1";"v0"
\ar "v0";"ww0"

\ar "ww0";"v1"
\ar "v1";"ww-2"
\ar "f0";"w-2"
\ar "f0";"ww-2"
\ar "w-2";"f2"
\ar "ww-2";"f2"

\ar "w1";"u2"
\ar "u2";"w-2"
\ar "e1";"f2"
\end{xy}
\end{equation}
By writing the ranks as $\mathcal{O}$-modules of vertices in $(\ref{main1})$, we obtain:
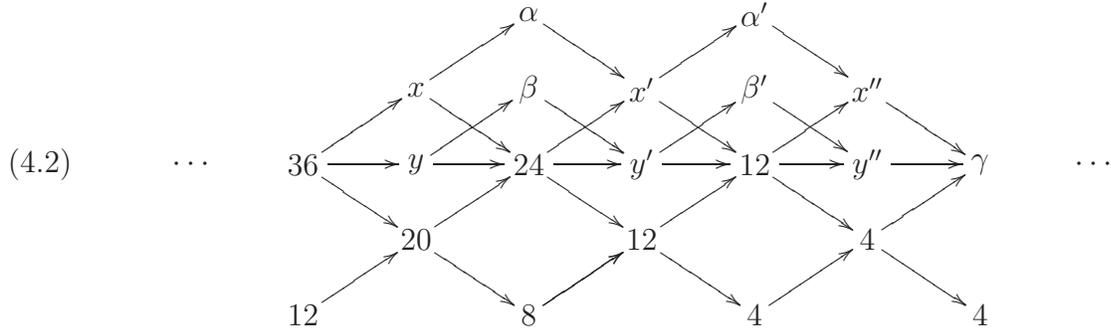
\begin{equation}\label{main2}
\begin{xy}
(-15,-20) *[o]+{8}="A",(15,-20)*[o]+{4}="B",
(45,-20)*[o]+{4}="C",(-45,-20)*[o]+{12}="z-2",
(-30,-10) *[o]+{20}="e-1",(0,-10)*[o]+{12}="e0",
(30,-10)*[o]+{4}="e1",
(-15,0) *[o]+{24}="f-1",(15,0)*[o]+{12}="f0",
(-45,0) *[o]+{36}="f-2",(45,0) *[o]+{\gamma}="f2",
(-30,10) *[o]+{x}="w0",(0,10)*[o]+{x'}="w1",(30,10)*[o]+{x''}="w-2",
(-30,0) *[o]+{y}="ww1",(0,0)*[o]+{y'}="ww0",(30,0)*[o]+{y''}="ww-2",
(-15,20) *[o]+{\alpha}="u1",(15,20) *[o]+{\alpha'}="u2",
(-15,10) *[o]+{\beta}="v0",(15,10) *[o]+{\beta'}="v1",(-60,0)*[o]+{\cdots}="dot1",(60,0)*[o]+{\cdots}="dot2",

\ar "e-1";"A"
\ar "A";"e0"
\ar "e0";"B"
\ar "B";"e1"
\ar "e1";"C"
\ar "A";"e0"
\ar "f-2";"e-1"
\ar "e-1";"f-1"
\ar "f-1";"e0"
\ar "e0";"f0"
\ar "f0";"e1"
\ar "z-2";"e-1"
\ar "f-2";"w0"
\ar "w0";"f-1"
\ar "f-1";"w1"
\ar "w1";"f0"
\ar "w0";"u1"
\ar "u1";"w1"

\ar "f-2";"ww1"
\ar "ww1";"f-1"
\ar "f-1";"ww0"
\ar "ww0";"f0"
\ar "ww1";"v0"
\ar "v0";"ww0"

\ar "ww0";"v1"
\ar "v1";"ww-2"
\ar "f0";"w-2"
\ar "f0";"ww-2"
\ar "w-2";"f2"
\ar "ww-2";"f2"

\ar "w1";"u2"
\ar "u2";"w-2"
\ar "e1";"f2"
\end{xy}
\end{equation}
Thus, we have the following system of linear equations:
\[ \left\{\begin{array}{lclc}
\beta +\beta' & = & y' &\cdots\cdots\cdots (1)\\
\alpha +\alpha' & = & x' &\cdots\cdots\cdots (2)\\
x +y & = & 40 &\cdots\cdots\cdots (3)\\
x' +y' & = & 24&\cdots\cdots\cdots (4) \end{array}\right. 
 \quad \left\{\begin{array}{lclc}
x +x' & = & 24+\alpha &\cdots\cdots\cdots (5)\\
y +y' & = & 24+\beta &\cdots\cdots\cdots (6)\\
x' +x'' & = & 12+\alpha' &\cdots\cdots\cdots (7)\\
y' +y'' & = & 12+\beta' &\cdots\cdots\cdots (8)
\end{array}\right. \]
From the equations (1), (2), (5) and (6), we have $x=24-\alpha'$ and $y=24-\beta'$. Using these equations and (3), we have $\alpha'+\beta'=8$.
On the other hand, the equations (4),(7), (8) and $\alpha'+\beta'=8$ imply $x''+y''=8$. 
Thus, we have $\gamma=0$, a contradiction.
Therefore, $\mathcal{CH}_\text{np}\neq \Z\widetilde{E}_6$.

Next we suppose that $\mathcal{CH}_\text{np}=\Z\widetilde{E}_7$. Then, $\mathcal{CH}_\text{np}$ has the following subquiver with upper bounds $U_n$:
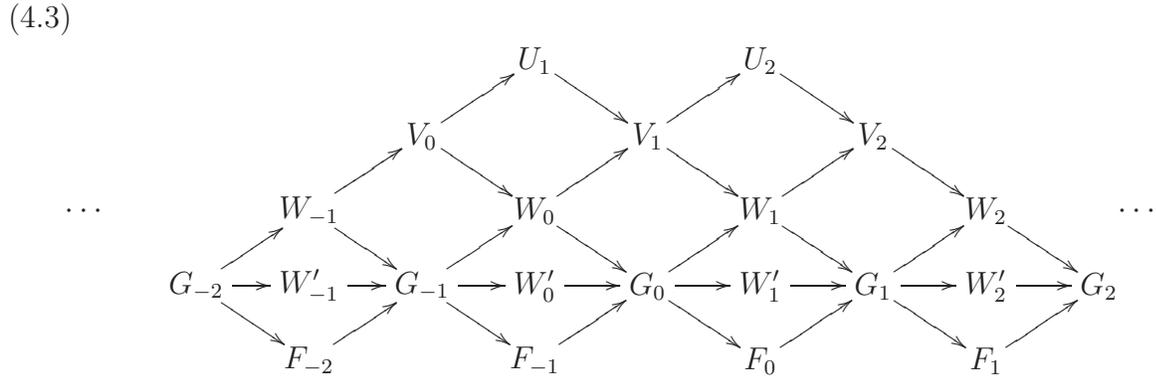
\begin{equation}\label{main3}
\begin{xy}
(0,-20) *[o]+{F_{-2}}="y1",(30,-20)*[o]+{F_{-1}}="y2",(60,-20) *[o]+{F_0}="y3",(90,-20) *[o]+{F_1}="y4",
(-15,-10) *[o]+{G_{-2}}="x1",(15,-10)*[o]+{G_{-1}}="x2",(45,-10)*[o]+{G_0}="x3",(75,-10)*[o]+{G_1}="x4",(105,-10)*[o]+{G_2}="x5",
(0,-10) *[o]+{W_{-1}'}="w1",(30,-10)*[o]+{W_0'}="w2",(60,-10)*[o]+{W_1'}="w3",(90,-10)*[o]+{W_2'}="w4",
(0,0) *[o]+{W_{-1}}="z1",(30,0)*[o]+{W_0}="z2",(60,0)*[o]+{W_1}="z3",(90,0)*[o]+{W_2}="z4"
,(15,10)*[o]+{V_0}="a2",(45,10)*[o]+{V_1}="a3",(75,10)*[o]+{V_2}="a4",
(30,20)*[o]+{U_1}="b2",(60,20)*[o]+{U_2}="b3",
(-30,0)*[o]+{\cdots}="dot1",(110,0)*[o]+{\cdots}="dot2",
\ar "x1";"y1"
\ar "x1";"w1"
\ar "x1";"z1"
\ar "y1";"x2"
\ar "w1";"x2"
\ar "z1";"x2"
\ar "x2";"y2"
\ar "x2";"w2"
\ar "x2";"z2"
\ar "y2";"x3"
\ar "w2";"x3"
\ar "z2";"x3"
\ar "x3";"y3"
\ar "x3";"w3"
\ar "x3";"z3"
\ar "y3";"x4"
\ar "w3";"x4"
\ar "z3";"x4"

\ar "z1";"a2"
\ar "a2";"z2"
\ar "z2";"a3"
\ar "a3";"z3"
\ar "z3";"a4"
\ar "a2";"b2"
\ar "b2";"a3"
\ar "a3";"b3"
\ar "b3";"a4"

\ar "x4";"y4"
\ar "y4";"x5"
\ar "x4";"w4"
\ar "w4";"x5"
\ar "x4";"z4"
\ar "z4";"x5"
\ar "a4";"z4"
\end{xy}
\end{equation}
By writing the ranks as $\mathcal{O}$-modules of vertices in $(\ref{main3})$, we obtain:
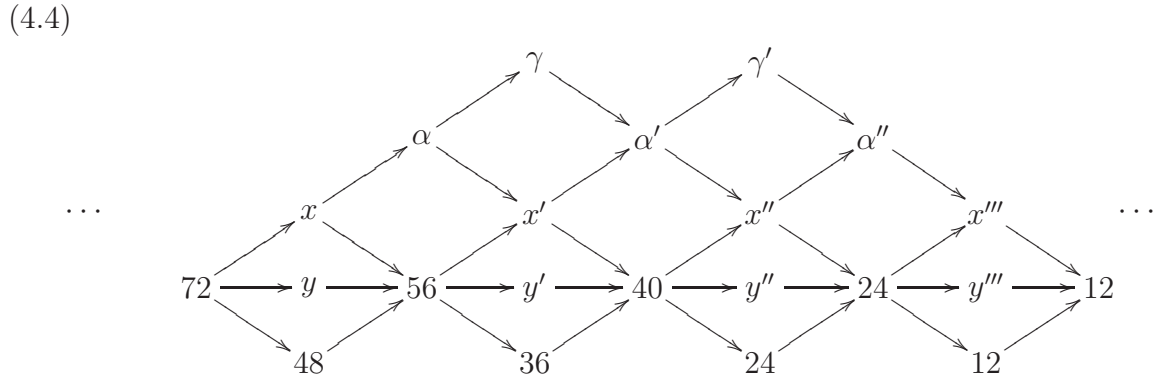
\begin{equation}\label{main4}
\begin{xy}
(0,-20) *[o]+{48}="y1",(30,-20)*[o]+{36}="y2",(60,-20) *[o]+{24}="y3",(90,-20) *[o]+{12}="y4",
(-15,-10) *[o]+{72}="x1",(15,-10)*[o]+{56}="x2",(45,-10)*[o]+{40}="x3",(75,-10)*[o]+{24}="x4",(105,-10)*[o]+{12}="x5",
(0,-10) *[o]+{y}="w1",(30,-10)*[o]+{y'}="w2",(60,-10)*[o]+{y''}="w3",(90,-10)*[o]+{y'''}="w4",
(0,0) *[o]+{x}="z1",(30,0)*[o]+{x'}="z2",(60,0)*[o]+{x''}="z3",(90,0)*[o]+{x'''}="z4"
,(15,10)*[o]+{\alpha}="a2",(45,10)*[o]+{\alpha'}="a3",(75,10)*[o]+{\alpha''}="a4",
(30,20)*[o]+{\gamma}="b2",(60,20)*[o]+{\gamma'}="b3",
(-30,0)*[o]+{\cdots}="dot1",(110,0)*[o]+{\cdots}="dot2",
\ar "x1";"y1"
\ar "x1";"w1"
\ar "x1";"z1"
\ar "y1";"x2"
\ar "w1";"x2"
\ar "z1";"x2"
\ar "x2";"y2"
\ar "x2";"w2"
\ar "x2";"z2"
\ar "y2";"x3"
\ar "w2";"x3"
\ar "z2";"x3"
\ar "x3";"y3"
\ar "x3";"w3"
\ar "x3";"z3"
\ar "y3";"x4"
\ar "w3";"x4"
\ar "z3";"x4"

\ar "z1";"a2"
\ar "a2";"z2"
\ar "z2";"a3"
\ar "a3";"z3"
\ar "z3";"a4"
\ar "a2";"b2"
\ar "b2";"a3"
\ar "a3";"b3"
\ar "b3";"a4"

\ar "x4";"y4"
\ar "y4";"x5"
\ar "x4";"w4"
\ar "w4";"x5"
\ar "x4";"z4"
\ar "z4";"x5"
\ar "a4";"z4"
\end{xy}
\end{equation}
where these unknown letters are the ranks of the corresponding vertices. Thus, we have the following system of linear equations by Proposition \ref{add}:
\[ \left\{\begin{array}{lclc}
x +y & = & 80 &\cdots\cdots\cdots (1)\\
x' +y' & = & 60 &\cdots\cdots\cdots (2)\\
x'' +y'' & = & 40 &\cdots\cdots\cdots (3) \\
x''' +y''' & = & 24&\cdots\cdots\cdots (4)\\
x +x' & = & 56+\alpha&\cdots\cdots\cdots (5)\\
x'+x'' &= & 40+\alpha'&\cdots\cdots\cdots (6) \\
x''+x''' &= & 24+\alpha''&\cdots\cdots\cdots (7) \end{array}\right. \quad
 \left\{\begin{array}{lclc}
y +y' & = & 56 &\cdots\cdots\cdots (8)\\
y' +y'' & = & 40 &\cdots\cdots\cdots (9)\\
y'' +y''' & = & 24 &\cdots\cdots\cdots (10) \\
x' +\gamma& = & \alpha+\alpha'&\cdots\cdots\cdots (11)\\
x''+\gamma' &= & \alpha'+\alpha''&\cdots\cdots\cdots (12) \\
\gamma+\gamma' &= & \alpha'&\cdots\cdots\cdots (13) \end{array}\right.\]
From the equations (1), (2), (5) and (8), we have $\alpha=28$. 
Similarly, the equations $(2),(3),(6)$ and $(9)$ yield $\alpha'=20$.
By adding both sides of the equations (11) and (12), we obtain the equation 
\[ x'+x''+\gamma +\gamma '=\alpha+2\alpha'+\alpha''. \]
From (6) and (13), the left hand side of the above equation is $40+2\alpha'$. 
Then, from (3), (4), (7), (10), we have
\[ 60=(x''+x''')+(y''+y''')=64, \]
a contradiction.

Finally, we assume that $\mathcal{CH}_\text{np}=\Z\widetilde{E}_8$. Then, $\mathcal{CH}_\text{np}$ has the following subquiver with upper bounds $V_n$ with $H(K_5)=5$:
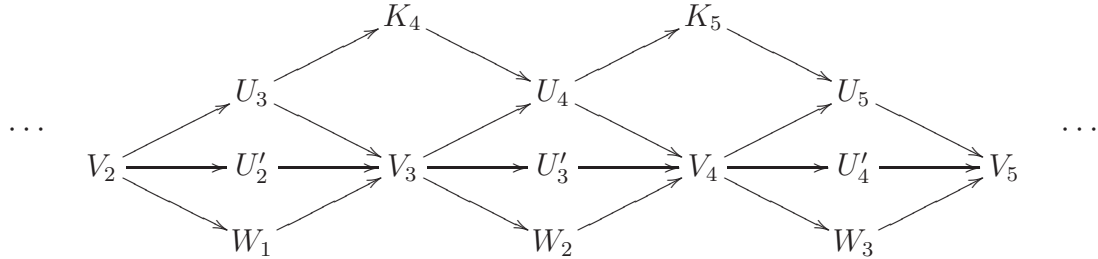
\begin{equation}\label{main5}
\begin{xy}
(0,-15) *[o]+{W_1}="y1",(40,-15)*[o]+{W_2}="y2",(80,-15) *[o]+{W_3}="y3",
(-20,-5) *[o]+{V_2}="x1",(20,-5)*[o]+{V_3}="x2",(60,-5)*[o]+{V_4}="x3",(100,-5)*[o]+{V_5}="x4",
(0,-5) *[o]+{U_2'}="w1",(40,-5)*[o]+{U_3'}="w2",(80,-5)*[o]+{U_4'}="w3",
(0,5) *[o]+{U_3}="z1",(40,5)*[o]+{U_4}="z2",(80,5)*[o]+{U_5}="z3",
(20,15)*[o]+{K_4}="a2",(60,15)*[o]+{K_5}="a3",
(-30,0)*[o]+{\cdots}="dot1",(110,0)*[o]+{\cdots}="dot2",
\ar "x1";"y1"
\ar "x1";"w1"
\ar "x1";"z1"
\ar "y1";"x2"
\ar "w1";"x2"
\ar "z1";"x2"
\ar "x2";"y2"
\ar "x2";"w2"
\ar "x2";"z2"
\ar "y2";"x3"
\ar "w2";"x3"
\ar "z2";"x3"
\ar "x3";"y3"
\ar "x3";"w3"
\ar "x3";"z3"
\ar "y3";"x4"
\ar "w3";"x4"
\ar "z3";"x4"

\ar "z1";"a2"
\ar "a2";"z2"
\ar "z2";"a3"
\ar "a3";"z3"
\end{xy}
\end{equation}
By writing the ranks as $\mathcal{O}$-modules of vertices in $(\ref{main5})$, we obtain
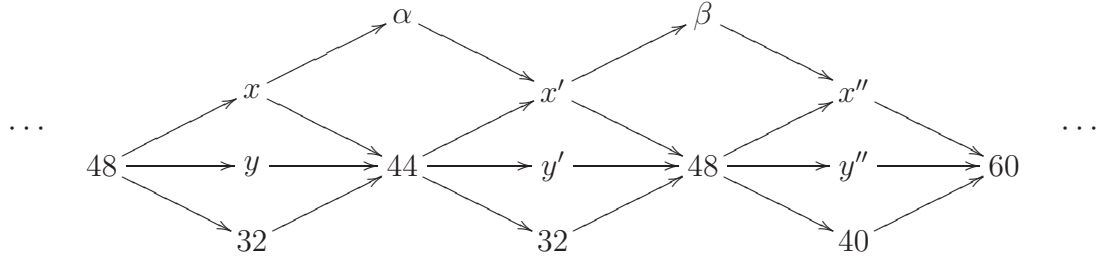
\begin{equation}\label{main6}
\begin{xy}
(0,-15) *[o]+{32}="y1",(40,-15)*[o]+{32}="y2",(80,-15) *[o]+{40}="y3",
(-20,-5) *[o]+{48}="x1",(20,-5)*[o]+{44}="x2",(60,-5)*[o]+{48}="x3",(100,-5)*[o]+{60}="x4",
(0,-5) *[o]+{y}="w1",(40,-5)*[o]+{y'}="w2",(80,-5)*[o]+{y''}="w3",
(0,5) *[o]+{x}="z1",(40,5)*[o]+{x'}="z2",(80,5)*[o]+{x''}="z3",
(20,15)*[o]+{\alpha}="a2",(60,15)*[o]+{\beta}="a3",(-30,0)*[o]+{\cdots}="dot1",(110,0)*[o]+{\cdots}="dot2",
\ar "x1";"y1"
\ar "x1";"w1"
\ar "x1";"z1"
\ar "y1";"x2"
\ar "w1";"x2"
\ar "z1";"x2"
\ar "x2";"y2"
\ar "x2";"w2"
\ar "x2";"z2"
\ar "y2";"x3"
\ar "w2";"x3"
\ar "z2";"x3"
\ar "x3";"y3"
\ar "x3";"w3"
\ar "x3";"z3"
\ar "y3";"x4"
\ar "w3";"x4"
\ar "z3";"x4"

\ar "z1";"a2"
\ar "a2";"z2"
\ar "z2";"a3"
\ar "a3";"z3"
\end{xy}
\end{equation}
such that these unknown values satisfy the following system of linear equations:
\[ \left\{\begin{array}{lclc}
x +y & = & 60 &\cdots\cdots\cdots (1)\\
x' +y' & = & 60 &\cdots\cdots\cdots (2)\\
x'' +y'' & = & 68 &\cdots\cdots\cdots (3) \\
x +x' & = & 44+\alpha&\cdots\cdots\cdots (4)
\end{array}\right. \quad
\left\{\begin{array}{lclc}
x'+x'' &= & 48+\beta&\cdots\cdots\cdots (5)\\
y +y' & = & 44 &\cdots\cdots\cdots (6)\\
y' +y'' & = & 48 &\cdots\cdots\cdots (7)\\
\alpha +\beta& = &x' &\cdots\cdots\cdots (8) 
\end{array}\right. \]
From $(1),(2),(4)$ and $(6)$, we obtain
\[ 120=x+x'+y+y'=88+\alpha, \]
and hence, $\alpha=32$. Similarly, using equations $(2),(3),(5)$ and $(7)$, we have $\beta=32$. The equation $(8)$ implies that $x'=64$, which contradicts with the equation $(2)$. Thus, the above system of linear equations has no solutions, and we conclude that $\mathcal{CH}_\text{np}\neq \Z\widetilde{E}_8$. Therefore, we have $\mathcal{CH}_\text{np}=\Z A_{\infty}$. 
\end{proof}

\section{Remarks on the shape of stable AR components}

In this section, we describe the shape of a component of the stable Auslander--Reiten quiver for a symmetric $\mathcal{O}$-order $A$. By Lemma \ref{valuation_of_loops}, non-periodic stable Auslander--Reiten components of $A$ have no loops. Thus, we can apply the Riedtmann structure theorem (Theorem \ref{Ried}) to such stable components.
Our goal is to show Propositions \ref{periodic_case} and \ref{no loop lemma2}. In this section, the middle term of the almost split sequence ending at $X$ is denoted by $E_X$.

\subsection{The case of periodic components}

Let $A$ be a symmetric $\mathcal{O}$-order and $\mathcal{C}$ a periodic component of the stable Auslander--Reiten quiver of $A$. Assume that the  stable Auslander--Reiten quiver $\Gamma_s(A)$ has infinitely many vertices.
In this subsection, we discuss the shape of  $\mathcal{C}$. 

\begin{proposition}[{\cite[Theorem 1.27]{AKM}}]\label{periodic_case}
Let $A$ be a symmetric $\O$-order and $\mathcal{C}$ a periodic component of $\Gamma_s(A)$. Assume that $\Gamma_s(A)$ has infinitely many vertices. Then, one of the following statements holds.
\begin{enumerate}[(1)]
\item If $\mathcal{C}$ has loops, then $\mathcal{C}\setminus\{\text{loops}\}=\mathbb{Z}A_\infty/\langle \tau \rangle$. Moreover, the loop appears on the boundary of $\mathcal{C}$.
\item If $\mathcal{C}$ has no loops, then $\mathcal{C}$ is of the form $\mathbb{Z}T/G$, where $T$ is a directed tree whose underlying graph is one of infinite Dynkin diagrams and $G$ is an admissible group.
\end{enumerate}
\end{proposition}
\begin{proof}
For each vertex $X\in\mathcal{C}$, we may choose $n_X\geq 1$ such that $\tau^{n_X}(X)\simeq X$. Define a $\mathbb{Q}_{\geq 0}$ valued function $f$ on $\mathcal{C}$ by
\[ f(X)=\dfrac{1}{n_X}\sum_{i=0}^{n_X-1}\mathrm{rank}\ \tau ^i(X). \]
Then, we have $f(X)=f(\tau X)$ for any $X$. By the definition of the Auslander--Reiten quiver of $A$, $\widetilde{\mathcal{C}}:=\mathcal{C}\setminus\{\text{loops}\}$ is a valued stable translation quiver. By applying Theorem \ref{Ried} to $\widetilde{\mathcal{C}}$, there are a directed tree $T$ and an admissible group $G$ such that $\widetilde{\mathcal{C}}=\mathbb{Z}T/G$. For $X\in T$, is is easily seen that
\[ \sum_{Y\to X} d_{YX}\mathrm{rank}\ Y \leq \mathrm{rank}\ X + \mathrm{rank}\ \tau (X), \]
which implies that $f$ satisfies
\begin{equation}\label{aaaaa}
 2f(X) \geq \sum_{Y\to X \text{ in }T}d_{YX}f(Y)+\sum_{X\to Y\text{ in }T}d'_{XY}f(Y),
 \end{equation}
for any $X\in T$. 
Suppose that $\mathcal{C}$ has no loops. Then, Theorem \ref{tree class} implies the statement (2).

Now, suppose that $\mathcal{C}$ has loops. Then, the inequality of (\ref{aaaaa}) is strict for some $X$. Since $\mathcal{C}$ has infinitely many vertices \cite[Proposition 1.26]{AKM}, the underlying tree $\overline{T}$ is $A_{\infty}$ by Theorem \ref{tree class}.
Therefore, $\tilde{\mathcal{C}}=\mathbb{Z}A_\infty/\langle \tau\rangle$ from Lemma \ref{valuation_of_loops}.
We may assume without loss of generality that $T$ is a chain of irreducible morphisms
\[ X_1 \to X_2 \to \cdots \to X_r \to \cdots. \]
Assume that $X_r$ has a loop for some $r$.

From now on, we prove that loops appear on the boundary of $\mathcal{C}$, that is, $r=1$. To obtain a contradiction, suppose that $r>1$. Then the almost split sequence starting at $X_r$ is
\[ 0 \longrightarrow X_r \longrightarrow X_r^{\oplus l} \oplus X_{r+1}\oplus X_{r-1} \longrightarrow X_r \longrightarrow 0 \]
where $l \geq 1$.  Since the subadditive function $f$ satisfies $f(X_t)\geq 1$  for all $t\geq 1$, we have
\[ f(X_r) \geq  (2- l)f(X_r) \geq f(X_{r+1}) + f(X_{r-1}) \geq f(X_{r+1})+1. \] 
We show that $f(X_m) \geq  f(X_{m+1})+1$ for $m \geq r$. Suppose that $f(X_{m-1}) \geq
f(X_m)+1$ holds. The same argument as above shows $2f(X_m) \geq f(X_{m-1}) +
f(X_{m+1})$, and the induction hypothesis implies $f(X_{m-1}) + f(X_{m+1}) \geq
f(X_m) + f(X_{m+1})+1$. Hence $f(X_m) \geq f(X_{m+1})+1$.
Thus, there exists a positive integer $t$ such that $f(X_t)<0$, a contradiction. Hence, $r=1$.
\end{proof}

\begin{corollary}\label{cor1}
Let $A$ be a symmetric $\O$-order, and let $\mathcal{C}$ be a periodic component of $\Gamma_{s}(A)$ with infinitely many vertices. If there exists a vertex $X$ of $\mathcal{C}$ such that the number of non-projective direct summands of $E_X$ is one, then $\mathcal{C}$ has no loops.
\end{corollary}

\begin{corollary}
Let $A$ be a symmetric $\O$-order, and let $\mathcal{C}$ be a periodic component of $\Gamma_{s}(A)$ with infinitely many vertices. If there exists a vertex $X$ of $\mathcal{C}$ such that 
\begin{enumerate}[(i)]
\item The number of non-projective indecomposable direct summands of $E_X$ is $1$. We denote by $Y$ the unique non-projective direct summand.
\item The number of non-projective indecomposable direct summands of $E_Y$ is $2$.
\end{enumerate}
Then, $\mathcal{C}$ is a tube.
\end{corollary}
\begin{proof}
Since $\Gamma_{s}(A)$ has infinitely many vertices, so is $\mathcal{C}$ by \cite[Proposition 1.26]{AKM}. By the assumption (i) and Corollary \ref{cor1}, $\mathcal{C}$ has no loops. Thus, the tree class $\overline{T}$  of $\mathcal{C}$ is one of infinite Dynkin diagrams. By the assumption (i), $\overline{T}\neq A^{\infty}_{\infty}$. By the assumption (ii), $\overline{T}\neq B_{\infty}$, $C_{\infty}$, $D_{\infty}$. Therefore, $\overline{T}$ is $A_{\infty}$. \end{proof}

\subsection{The case of non-periodic components}

Let $\mathcal{C}$ be a component of $\Gamma_{s}(A)$. Recall the function $D:\mathcal{C}_0\to \mathbb{Z}_{\geq 0}$ defined by
\[ D(X):=\sharp\{ \text{non-projective indecomposable direct summands of $X\otimes\kappa\}$}. \] 

\begin{proposition}\label{no loop lemma2}
Let $A$ be a symmetric $\O$-order, and let $\mathcal{C}$ be a non-periodic component of the stable Auslander--Reiten quiver of $A$. 
Assume either 
\begin{enumerate}[(i)]
\item $\mathcal{C}$ does not contain Heller lattices or 
\item $A\otimes\kappa$ has finite representation type.
\end{enumerate}
Then, the tree class of $\mathcal{C}$ is one of infinite Dynkin diagrams or Euclidean diagrams.
\end{proposition}
\begin{proof}

Since $\mathcal{C}$ has no loops, there exist a directed tree $T$ and an admissible group $G$ such that $\mathcal{C}\simeq \mathbb{Z}T/G$ by Theorem \ref{Ried}. Suppose that  $\mathcal{C}$ does not contain Heller lattices. In this case, the function $D$ is additive with $D(X)=D(\tau X)$, for all $X\in\mathcal{C}$ by Lemmas \ref{stable of D} and \ref{additive of D}. Then, for all $X\in T$, we have
\[ 2D(X) = \sum_{Y\to X \text{ in }T}d_{YX}D(Y)+\sum_{X\to Y\text{ in }T}d'_{XY}D(Y). \]
Therefore, it follows from Theorem \ref{tree class} that $\overline{T}$ is one of infinite Dynkin diagrams or Euclidean diagrams.

Suppose that $A\otimes\kappa$ has finite representation type. Since the number of isoclasses of Heller lattices is finite, there exists an integer $n_X$ such that both $\tau ^{n_X}X$ and $\tau ^{n_X+1}X$ are not Heller lattices for any vertex $X\in\mathcal{C}$. Thus, $D$ is an additive function with $D=D\circ\tau$ on $\mathcal{C}$.
\end{proof}


\bibliographystyle{amsalpha}

\end{document}